\documentclass[a4paper,11pt]{amsart}
\usepackage{amssymb}
\usepackage{amsmath}
\usepackage{amsthm}
\usepackage{verbatim}
\usepackage{url}
\usepackage[all]{xy}

\theoremstyle{plain}

\newtheorem{theorem}{Theorem}[section]

\newtheorem{proposition}[theorem]{Proposition}
\newtheorem{lemma}[theorem]{Lemma}
\newtheorem{corollary}[theorem]{Corollary}

\newtheorem*{property*}{}

\theoremstyle{definition}

\newtheorem{definition}[theorem]{Definition}
\newtheorem{example}[theorem]{Example}
\newtheorem{construction}[theorem]{Construction}

\newtheorem{subsec}[theorem]{}

\theoremstyle{remark}
\newtheorem{remark}[theorem]{Remark}

\newcommand{\C}{{\mathbf{C}}}
\newcommand{\R}{{\mathbf{R}}}
\newcommand{\Q}{{\mathbf{Q}}}
\newcommand{\Z}{{\mathbf{Z}}}

\newcommand{\G}{{\mathbb{G}}}
\newcommand{\ppP}{\mathbb{P}}
\newcommand{\A}{{\mathbb{A}}}
\newcommand{\PP}{{\mathbb P}}

\newcommand{\X}{{{\sf X}}}

\newcommand{\AAA}{{\mathsf A}}
\newcommand{\BB}{{\mathsf B}}
\newcommand{\CC}{{\mathsf C}}
\newcommand{\DD}{{\mathsf D}}
\newcommand{\EE}{{\mathsf E}}
\newcommand{\FF}{{\mathsf F}}
\newcommand{\GG}{{\mathsf G}}

\newcommand{\Aut}{\mathrm{Aut}}
\newcommand{\SAut}{{\rm SAut}}

\newcommand{\im}{\mathrm{im}}

\newcommand{\inn}{\mathrm{inn}}

\newcommand{\Spec}{{\rm Spec\,}}
\newcommand{\Gal}{{\rm Gal}}

\newcommand{\BRD}{{\rm BRD}}
\newcommand{\Hom}{{\rm Hom}}
\newcommand{\Dyn}{{\rm Dyn}}
\newcommand{\id}{{\rm id}}

\newcommand{\Val}{{\rm Val}}
\newcommand{\val}{{\rm val}}

\newcommand{\Stab}{{\rm Stab}}
\newcommand{\CF}{{\rm CF}}

\newcommand{\charr}{{\rm char\,}}

\newcommand{\into}{\hookrightarrow}
\newcommand{\isoto}{\overset{\sim}{\to}}

\newcommand{\labelto}[1]{\xrightarrow{\makebox[1.5em]{\scriptsize ${#1}$}}}

\newcommand{\sX}{{\mathcal X}}
\newcommand{\sV}{{\mathcal V}}
\newcommand{\sP}{{\mathcal P}}
\newcommand{\sD}{{\mathcal{D}}}
\newcommand{\sC}{{\mathcal C}}
\newcommand{\sF}{{\mathcal F}}
\newcommand{\sA}{{\mathcal A}}
\newcommand{\sN}{{\mathcal{N}}}
\newcommand{\sU}{{\mathcal{U}}}

\newcommand{\NGH}{{\sN_G(H)}}

\newcommand{\vs}{\varsigma}
\newcommand{\ve}{{\varepsilon}}
\newcommand{\vphi}{{\varphi}}
\newcommand{\vk}{\varkappa}

\newcommand{\veg}{{\ve_\gamma}}
\newcommand{\atil}{{\widetilde{a}}}
\newcommand{\kinf}{{k\cup\{\infty\}}}
\newcommand{\gams}{{\gamma_*}}

\newcommand{\uY}{{\, \underset{Y}{\cdot}\, }}
\newcommand{\uZ}{{\, \underset{Z}{\cdot}\, }}

\newcommand{\ttau}{{\raisebox{.0ex}{$^\tau$}}}
\newcommand{\uT}{{\, \underset{\ttau}{\cdot}\, }}

\newcommand{\Gtil}{{\widetilde{G}}}

\newcommand{\Omt}{{\Omega^{(2)}}}
\newcommand{\Omone}{{\Omega^{(1)}}}

\newcommand{\St}{{S_2}}

\newcommand{\GL}{{\rm GL}}
\newcommand{\SO}{{\rm SO}}
\newcommand{\PGL}{{\rm PGL}}
\newcommand{\SL}{{\rm SL}}

\newcommand{\Sp}{{\rm Sp}}

\newcommand{\Ggg}{{\mathfrak{g}}}
\newcommand{\hh}{{\mathfrak{h}}}
\newcommand{\nn}{{\mathfrak{n}}}
\newcommand{\uu}{{\mathfrak u}}

\newcommand{\Ddim}{{\rm dim}}
\newcommand{\Lie}{{\rm Lie}}

\newcommand{\barH}{{\overline{H}}}

\newcommand{\stY}{\smash{\widetilde Y}}
\newcommand{\stH}{\smash{\widetilde H}}
\newcommand{\stsV}{\smash{\widetilde \sV}}
\newcommand{\stsX}{\smash{\widetilde \sX}}
\newcommand{\stSigma}{\smash{\widetilde \Sigma}}
\newcommand{\stsD}{\smash{\widetilde \sD}}

\newcommand{\Ank}{{\A^n_k}}
\newcommand{\Anknull}{{\A^n_{k_0}}}

\newcommand{\Map}{{\rm Map}}
\newcommand{\SigmaN}{\Sigma^N}

\newcommand{\hs}{\kern 0.8pt}
\newcommand{\hm}{\kern -0.6pt}

\newcommand{\emm}{\bfseries}

\begin{document}

\title[Equivariant models of spherical varieties]
{Equivariant models of spherical varieties%
}

\author{Mikhail Borovoi\\{\Tiny with an appendix by}\\ Giuliano Gagliardi}

\address{Borovoi: Raymond and Beverly Sackler School of Mathematical Sciences,
Tel Aviv University, 6997801 Tel Aviv, Israel}
\email{borovoi@post.tau.ac.il}

\address{Gagliardi: Raymond and Beverly Sackler School of Mathematical Sciences,
Tel Aviv University, 6997801 Tel Aviv, Israel}
\email{giulianog@mail.tau.ac.il}

\thanks{This research was partially supported by the Hermann Minkowski Center for Geometry and by the Israel Science Foundation (grant No. 870/16)}




\begin{abstract}
Let $G$ be a connected semisimple group over an algebraically closed field $k$ of characteristic 0.
Let $Y=G/H$ be a spherical homogeneous space of $G$, and let $Y'$ be a spherical embedding of $Y$.
Let $k_0$ be a subfield of $k$.
Let $G_0$ be a $k_0$-model ($k_0$-form) of $G$.
We show that if $G_0$ is an {\em inner} form of a split group and if the subgroup $H$ of $G$ is spherically
closed, then $Y$ admits a $G_0$-equivariant $k_0$-model.
If we replace the assumption that $H$ is spherically closed
by the stronger assumption that $H$ coincides with its normalizer in $G$,
then  $Y$ and $Y'$ admit compatible  $G_0$-equivariant $k_0$-models, and these models are unique.
\end{abstract}

\maketitle

\tableofcontents

\setcounter{section}{-1}

\section{Introduction}

Let $k$ be an algebraically closed field of characteristic 0, and let $k_0\subset k$ be a subfield.
Let $X$ be a $k$-variety, that is, an algebraic variety over $k$.
By a {\em $k_0$-model} ($k_0$-form) of $X$ we mean a $k_0$-variety $X_0$ together with an isomorphism of $k$-varieties
\[\vk_X\colon X_0\times_{k_0}k \isoto X.\]

Let $G$ be a connected semisimple group over $k$.
Let $Y$ be a $G$-variety, that is, an algebraic variety over $k$ together with a morphism
\[\theta \colon G\times_k Y\to Y\]
defining an action of $G$ on $Y$.
We say that $(Y,\theta)$ is a {\em $G$-$k$-variety} or just that $Y$ is a $G$-$k$-variety.

Let $G_0$ be a {\em $k_0$-model} ($k_0$-form) of $G$, that is, an algebraic group over $k_0$
together with an isomorphism of algebraic $k$-groups
\[\vk_G\colon G_0\times_{k_0} k\isoto G.\]
By a {\em  $G_0$-equivariant $k_0$-model of the $G$-$k$-variety $(Y,\theta)$} we mean a $G_0$-$k_0$-variety $(Y_0,\theta_0)$
together with an {\em isomorphism of $G$-$k$-varieties} $\vk_Y\colon Y_0\times_{k_0} k\isoto Y$,
that is, an isomorphism of $k$-varieties $\vk_Y$ such that the following diagram commutes:
\begin{equation}\label{e:model-action}
\begin{aligned}
\xymatrix{
G_{0,k}\times_k Y_{0,k}\ar[r]^-{\theta_{0,k}}\ar[d]_{\vk_G\times\vk_Y}   & Y_{0,k}\ar[d]^{\vk_Y} \\
G\times_k Y\ar[r]^-\theta                                                &Y
}
\end{aligned}
\end{equation}
where $G_{0,k}:=G_0\times_{k_0} k$ and  $Y_{0,k}:=Y_0\times_{k_0} k$.
For a given $k_0$-model $G_0$  of $G$ we ask whether there exists a  $G_0$-equivariant  $k_0$-model $Y_0$ of $Y$.

From now on till the end of the Introduction we assume that  $Y$ is a {\em spherical homogeneous space of $G$.}
This means that $Y=G/H$ (with the natural action of $G$) for some algebraic subgroup $H\subset G$
and that a Borel subgroup $B$ of $G$ has an open orbit in $Y$.
Then the set of orbits of $B$ in $Y$ is finite; see, for example, Timashev \cite[Section 25.1]{Timashev}.

Let $Y\into Y'$ be a {\em spherical embedding of \  $Y=G/H$.}
This means that $Y'$ is a $G$-$k$-variety, that $Y'$ is a normal variety, and that $Y'$ contains $Y$ as an open dense $G$-orbit.
Then $B$ has an open dense orbit in $Y'$.
Moreover the set of orbits of $B$ (and hence, of $G$) in $Y'$ is finite;  see,  for example, Timashev \cite[Section 25.1]{Timashev}.

Inspired by the works of Akhiezer and Cupit-Foutou  \cite{ACF}, \cite{Akhiezer}, \cite{CF},
for a given $k_0$-model $G_0$  of $G$ we ask whether there exist a $G_0$-equivariant $k_0$-model $Y_0$ of $Y$
and a $G_0$-equivariant $k_0$-model $Y'_0$ of $Y'$.

Since $ \charr k=0$, by a result of Alexeev and Brion \cite[Theorem 3.1]{AB},
see Knop's MathOverflow answer \cite{Knop-MO-AB} and  Appendix \ref{s:app-AB} below,
the spherical subgroup $H$ of $G$ is conjugate to some (spherical) subgroup
defined over the algebraic closure of $k_0$ in $k$.
Therefore, from now on we assume that $k$ is an algebraic closure of $k_0$.
We set $\Gamma=\Gal(k/k_0)$ (the Galois group of $k$ over $k_0$).

Let $T$ be a maximal torus of $G$ contained in a Borel subgroup $B$.
We consider the Dynkin diagram $\Dyn(G)=\Dyn(G,T,B)$.
The $k_0$-model $G_0$ of $G$ defines the so-called $\ast$-action of $\Gamma=\Gal(k/k_0)$
on the Dynkin diagram $\Dyn(G)$; see  Tits \cite[Section 2.3, p.\,39]{Tits}.
In other words, we obtain a homomorphism
\[\ve\colon \Gamma\to\Aut\,\Dyn(G).\]
The $k_0$-group $G_0$ is called an {\em inner form} (of a split group)
if the $\ast$-action is trivial, that is, if $\ve_\gamma=\id$ for all $\gamma\in\Gamma$.
For example, if $G$ is a simple group of any of the types
$\AAA_1,\ \BB_n,\ \CC_n,\ \EE_7,\ \EE_8,\ \FF_4,\ \GG_2$,
then any $k_0$-model $G_0$ of $G$ is an inner form, because in these cases
$\Dyn(G)$ has no nontrivial automorphisms.
If $G_0$ is a split $k_0$-group, then of course $G_0$ is an inner form.

Let $\sD(Y)$ denote the  {\em set of colors of $Y=G/H$}, that is, the (finite) set
of the closures of $B$-orbits of codimension one in $Y$.
A spherical subgroup $H\subset G$ is called {\em spherically closed}
if the automorphism group $\Aut^G(Y)=\NGH/H$
acts on $\sD=\sD(Y)$ faithfully, that is, if the homomorphism
\begin{equation*}
\Aut^G(Y)\to\Aut(\sD)
\end{equation*}
is injective.
Here $\NGH$ denotes the normalizer of $H$ in $G$, and $\Aut(\sD)$
denote the group of permutations of the finite set $\sD$.

\begin{example}\label{ex:CF}
Let $k=\C$, \  $G={\rm PGL}_{2,\C}$, $H=T$ (a maximal torus), $Y=G/T$.
Then $|\sN_G(T)/T|=2$, and the spherical homogeneous space $Y$ of $G$  has exactly two colors,
which are swapped by the non-unit element of $\sN_G(T)/T$.
We see that the subgroup $H=T$ of $G$ is spherically closed.
\end{example}

\begin{theorem}\label{t:CF}
Let $G$ be a connected semisimple group over an algebraically closed field $k$ of characteristic 0.
Let $Y=G/H$ be a spherical homogeneous space of $G$.
Let $k_0$ be a subfield of $k$ such that $k$ is an algebraic closure of $k_0$.
Let $G_0$ be a $k_0$-model of $G$.
Assume that:
\begin{enumerate}
\item[\rm(i)] $G_0$ is an inner form, and
\item[\rm(ii)] $H$ is spherically closed.
\end{enumerate}
Then $Y$ admits a  $G_0$-equivariant $k_0$-model $Y_0$.
\end{theorem}

Theorem \ref{t:CF} (which  was inspired by Theorem 1.1 of Akhiezer \cite{Akhiezer}
and by Corollary 1 of Cupit-Foutou \cite[Section 2.5]{CF}), is a special case
of the more general Theorem \ref{t:CF-CF} below,
where instead of assuming that $G_0$ is an inner form,
we assume only that for all $\gamma\in\Gamma$ the automorphism $\ve_\gamma$ of $\Dyn(G)$
preserves the {\em combinatorial invariants}
(Luna-Losev invariants) of the spherical homogeneous space $Y$.
This assumption is necessary for the existence of a $G_0$-equivariant $k_0$-model of $Y$;
see Proposition \ref{p:Huruguen} below.

\begin{remark}
Necessary and sufficient conditions for the existence of  $G_0$-equivariant $k_0$-model of $Y=G/H$
were given by Moser-Jauslin and Terpereau  \cite[Theorem 3.18]{MJT}
in the case when $k_0=\R$ and $H$ is a {\em horospherical} subgroup of $G$.
Note that a horospherical subgroup $H\subset G$ is not spherically closed unless it is parabolic, in which case $\sN_G(H)=H$.
The general case, when $k_0$ is an arbitrary field of characteristic 0
and $H$ is an arbitrary spherical subgroup of $G$,
will be treated in the forthcoming article \cite{BG} of the author and G. Gagliardi.
We have to assume that $\charr k=0$ when dealing with spherical varieties,
because  we use Losev's uniqueness theorem \cite[Theorem 1]{Losev},
which has been proved only in characteristic 0.
\end{remark}

Note that in general a $G_0$-equivariant $k_0$-model $Y_0$ in Theorem \ref{t:CF} is not unique.
The following theorem is a special case of the more general theorem \ref{t:Boris-Boris} below.

\begin{theorem}\label{t:Boris}
In Theorem \ref{t:CF} the set of isomorphism classes of $G_0$-equivariant $k_0$-models of $Y=G/H$
is naturally a principal homogeneous space of the finite abelian group
\[ H^1(\Gamma,\Aut^G(Y))\simeq\Map(\Omt,\hs \Hom(\Gamma,\St)\hs)\]
Here $\St$ is the symmetric group on two symbols,
$\Omt=\Omt(Y)$ is the finite set defined
in Section \ref{s:invariants} below (before Definition \ref{d:comb-inv}),
and $\Map(\Omt,\hs \Hom(\Gamma,\St)\hs)$ denotes the group of maps
from the set $\Omt$ to the abelian group  $\Hom(\Gamma,\St)$.
\end{theorem}

In particular, for $k_0=\R$ we have $\Hom(\Gamma,\St)\cong\St$,
and therefore, the number of these isomorphism classes is $2^s$, where $s=|\Omt|$.
For $G$ and $Y$ as in Example \ref{ex:CF} we have $s=1$, hence for each of the two $\R$-models of $G$
there are exactly two non-isomorphic equivariant $\R$-models of $Y$; see Example \ref{ex:SO3} below.

\begin{corollary}[\rm Akhiezer's theorem]
\label{c:Akhiezer}
In Theorem \ref{t:CF}, instead of {\em (ii)} assume that
\begin{enumerate}
\item[\rm(ii$'$)] $H$ is self-normalizing, that is, $\NGH=H$.
\end{enumerate}
Then $Y=G/H$ admits a $G_0$-equivariant $k_0$-model $Y_0$,
and this model is unique up to a unique isomorphism.
\end{corollary}

Indeed, since $H$ is self-normalizing, we have $\Aut^G(Y)=\NGH/H=\{1\}$,
and  hence, $H$  is spherically closed.
By Theorem \ref{t:CF}, $Y$ admits a  $G_0$-equivariant $k_0$-model.
The uniqueness assertion is obvious because $\Aut^G(Y)=\{ 1\}$.

Corollary \ref{c:Akhiezer}  generalizes Theorem 1.1 of Akhiezer \cite{Akhiezer},
where the case $k_0=\R$ was considered.

\begin{theorem}\label{t:Huruguen}
Under the assumptions of Corollary \ref{c:Akhiezer},
 {\emm any} spherical embedding $Y'$ of $Y=G/H$ admits a  $G_0$-equivariant  $k_0$-model $Y'_0$.
This $k_0$-model $Y'_0$ is compatible with the unique  $G_0$-equivariant
$k_0$-model $Y_0$ of $Y$ from Corollary \ref{c:Akhiezer},
and hence is unique up to a unique isomorphism.
\end{theorem}

Theorem \ref{t:Huruguen} generalizes Theorem 1.2 of Akhiezer \cite{Akhiezer},
who proved in the case $k_0=\R$ that
the {\em wonderful} embedding of $Y$ admits a unique  $G_0$-equivariant $\R$-model.
Our proof of Theorem \ref{t:Huruguen} uses results of Huruguen \cite{Huruguen}.
Note that in Theorem \ref{t:Huruguen} we do not assume that $Y'$ is quasi-projective.

Theorems \ref{t:CF}, \ref{t:Boris}, and \ref{t:Huruguen} seem to be new even in the case $k_0=\R$.

The plan of the rest of the article is as follows.
In Sections \ref{s:conj}--\ref{s:models} we consider models and  semilinear morphisms
for general $G$-varieties and  homogeneous spaces of $G$, not necessarily spherical.
In Sections \ref{s:invariants}--\ref{s:action} we consider
combinatorial invariants of  spherical homogeneous spaces.
Following  ideas of Akhiezer \cite[Theorem 1.1]{Akhiezer} and Cupit-Foutou \cite[Theorem 3(1), Section 2.2]{CF},
for $\gamma\in\Gamma=\Gal(k/k_0)$
we give a criterion of isomorphism of a spherical homogeneous space $Y=G/H$
and the ``conjugated'' variety $\gamma_* Y=G/\gamma(H)$
in terms of the action of $\gamma$ on the combinatorial invariants of $G/H$.
In Sections \ref{s:aut-free}--\ref{s:embeddings} we prove
Corollary \ref{c:Akhiezer}, Theorem \ref{t:CF}, Theorem \ref{t:Boris}, and Theorem \ref{t:Huruguen}.
In Appendix \ref{s:app-AB}, for a connected  reductive group $G_0$
defined over an {\em algebraically closed} field $k_0$ of  characteristic 0
and for an algebraically closed extension $k\supset k_0$,
it is proved that any spherical subgroup $H$ of the base change $G=G_0\times_{k_0} k$
is conjugate to a (spherical) subgroup defined over $k_0$.
In Appendix \ref{s:App},  following Friedrich Knop's MathOverflow answer \cite{Knop-MO} to the author's question,
Giuliano Gagliardi gives a proof of an unpublished theorem of Ivan Losev describing the image
of $\Aut^G(G/H)=\NGH/H$ in the group of permutations of $\sD(G/H)$.
Our proofs of Theorems \ref{t:CF},  \ref{t:Boris}, and \ref{t:Huruguen} use this result of Losev.

{\em Acknowledgements.}
The author is very grateful to Friedrich Knop for answering the author's numerous MathOverflow questions,
especially for the answer \cite{Knop-MO},
to Giuliano Gagliardi for writing Appendix \ref{s:App},
and to Roman Avdeev for suggesting Example \ref{ex:Avdeev} and proving Proposition \ref{p:Avdeev-bis}.
It is a pleasure to thank Michel Brion for very helpful e-mail correspondence.
The author thanks Dmitri Akhiezer, St\'ephanie Cupit-Foutou,
Cristian D. Gonz\'alez-Avil\'es, David Harari,  Boris Kunyavski\u\i, and Stephan Snegirov  for helpful discussions.
The author thanks the referees for careful reading the article and very useful comments,
which helped to improve the exposition.
This article was written during the author's visits
to the University of La Serena (Chile) and to the Paris-Sud University,
and he is grateful to the departments of mathematics of these universities
for support and excellent working conditions.

\noindent
{\em Notation and assumptions.}

$k$ is a field. In  Section \ref{s:G} and everywhere starting Section \ref{s:homog}, $k$ is algebraically closed.
Starting Section \ref{s:invariants} we assume that $\charr k=0$.

$k_0$ is a subfield of the algebraically closed field $k$ such that $k$ is a Galois extension of $k_0$
(except for Appendix  \ref{s:app-AB}), hence $k_0$ is perfect.

A {\em $k$-variety} is a geometrically  reduced separated scheme of finite type over $k$, not necessarily irreducible.

An {\em algebraic $k$-group} is a smooth  $k$-group scheme of finite type over $k$, not necessarily connected.
All {algebraic $k$-subgroups} are assumed to be smooth.
Starting Section \ref{s:homog},  all algebraic groups are assumed to be linear (affine).

\section{Semi-morphisms of $k$-schemes}
\label{s:conj}

\begin{subsec}
Let $k$ be a field and let $\Spec k$ denote the spectrum of $k$.
By a {\em $k$-scheme} we mean a pair $(Y,p_Y)$, where $Y$ is a scheme
and $p_Y\colon Y\to \Spec k$ is a morphism of schemes.
Let $(Y,p_Y)$ and $(Z,p_Z)$ be two $k$-schemes.
By a {\em $k$-morphism,} or a {\em morphism of $k$-schemes,}
\[ \lambda\colon \ (Y,p_Y)\to (Z,p_Z)\]
we mean a morphism of schemes $\lambda\colon Y\to Z$  such that the following diagram commutes:
\begin{equation*}
\xymatrix@R=25pt@C=40pt{
Y\ar[r]^\lambda \ar[d]_{p_Y}          & Z\ar[d]^{p_Z} \\
\Spec k\ar[r]^-{\id}   & \Spec k
}
\end{equation*}

Let  $\gamma\colon k\to k$ be an automorphism of $k$ (we write  $\gamma\in\Aut(k)$\hs).
Let
\[\gamma^*:=\Spec\gamma\colon\ \Spec k\to\Spec k \]
denote the induced automorphism of $\Spec k$; then $(\gamma\gamma')^*=(\gamma')^*\circ\gamma^*$.

Let $(Y,p_Y$) be a $k$-scheme.
By abuse of notation we write just that $Y$ is a $k$-scheme.
We define the {\em $\gamma$-conjugated $k$-scheme} $\gamma_*(Y,p_Y)=(\gamma_*Y,\gamma_* p_Y)$
to be the base change of $(Y,p_Y)$ from $\Spec k$ to $\Spec k$ via $\gamma^*$.
By abuse of notation we write just $\gamma_* Y$  for $\gamma_*(Y,p_Y)$.
\end{subsec}

\begin{lemma}
Let $(Y,p_Y)$ be a $k$-scheme, and let  $\gamma\in\Aut(k)$.
Then the  $\gamma$-conjugat\-ed $k$-scheme $\gamma_*(Y,p_Y)$
is canonically isomorphic to $(Y,(\gamma^*)^{-1}\circ p_Y)$
\end{lemma}

\begin{proof}
Write $(X,p_X)=\gamma_*(Y,p_Y)$; then $X$ comes with a canonical morphism $\lambda\colon X\to Y$
such that the following diagram commutes:
\[
\xymatrix{
X\ar[r]^{\lambda}\ar[d]_{p_{X}} &Y\ar[d]^{p_Y}\\
\Spec k\ar[r]^{\gamma^*}  &\Spec k
} \]
Since $(\gamma^{-1})_*(\gamma_*(Y,p_Y))$ is canonically isomorphic to $(Y,p_Y)$,
one can easily see that $\lambda$ is an isomorphism of schemes.
From the above diagram we obtain a commutative diagram
\[
\xymatrix@R=20pt@C=50pt{
X\ar[r]^{\lambda} \ar[dd]_{p_X}        &Y\ar[d]^-{p_Y} \\
                                 &\Spec k\ar[d]^-{(\gamma^*)^{-1}} \\
\Spec k\ar[r]^{\id}   &\Spec k
}
\]
which gives a canonical isomorphism of $k$-schemes $(X,p_X)\isoto (Y,(\gamma^*)^{-1}\circ p_Y)$.
 \end{proof}

\begin{subsec}
We define an action of  $\gamma\colon k\to k$ on $k$-points.
Let $y$ be a $k$-point of $Y$, that is, a morphism
$y\colon \Spec k\to Y$ such that $p_Y\circ y=\id_{\Spec k}$\hs.
We denote
\begin{equation}\label{e:y-!}
\gamma_!(y)=y\circ \gamma^*\colon \ \Spec k\to\Spec k\to Y;
\end{equation}
then an easy calculation shows that
$\gamma_!(y)$ is a $k$-point of $\gamma_* Y$,
where we identify $\gamma_*(Y,p_Y)$ with $(Y,(\gamma^*)^{-1}\circ p_Y)$.
Thus we obtain a bijection
\begin{equation}\label{e:gamma-!}
\gamma_!\colon Y(k)\to (\gamma_* Y)(k),\quad  y\mapsto \gamma_! (y).
\end{equation}
\end{subsec}

\begin{subsec}
Let $G$ be a $k$-group scheme.
Following Flicker, Scheiderer, and Sujatha \cite[(1.2)]{FSS},
we define the $k$-group scheme $\gamma_* G$
to be the base change of $G$ from $\Spec k$ to $\Spec k$ via $\gamma^*$.
Then the map \eqref{e:gamma-!}
\[ \gamma_!\colon G(k)\to (\gamma_*G)(k)\]
is an isomorphism of abstract groups (because for any field extension $\lambda\colon k\into k'$,
the corresponding map on rational points
\[\lambda_!\colon G(k)\to (G\times_k k')(k')\]
is a homomorphism).
If $H\subset G$ is a $k$-group subscheme, then $\gamma_* H$ is naturally a $k$-group subscheme of $\gamma_* G$
(because a base change of a group subscheme is a group subscheme).
From the commutative diagram
\[
\xymatrix{
H(k)\ar[r]^-{\gamma_!}\ar@{_(->}[d]   & (\gamma_* H)(k)\ar@{_(->}[d]  \\
G(k)\ar[r]^-{\gamma_!}         & (\gamma_* G)(k)
}
\]
we see that
\begin{equation*}
(\gamma_* H)(k)=\gamma_!(H(k))\subset(\gamma_* G)(k).
\end{equation*}

Let $(Y, \theta)$ be a {\em $G$-$k$-scheme} (a $G$-scheme over $k$), where
\[\theta\colon G\times_k Y\to Y,\]
is an action of $G$ on $Y$.
 By abuse of notation we write just that $Y$ is a $G$-$k$-scheme.
We  define the $\gamma_* G$-$k$-scheme $\gamma_*(Y,\theta)=(\gamma_* Y,\gamma_*\theta)$
to be the base change of $(Y,\theta)$ from $\Spec k$ to $\Spec k$ via $\gamma^*$.
\end{subsec}

\begin{definition}
Let $(Y,p_Y)$ and $(Z,p_Z)$ be two $k$-schemes.
A {\em semilinear morphism}
\[ (\gamma,\nu)\colon \ (Y,p_Y)\to (Z,p_Z)\]
is a pair $(\gamma,\nu)$,
where $\gamma\colon k\to k$ is an automorphism of $k$ and
$\nu\colon Y\to Z$ is a morphism of schemes, such that the following diagram commutes:
\begin{equation*}
\xymatrix@R=25pt@C=40pt{
Y\ar[r]^\nu \ar[d]_{p_Y}          & Z\ar[d]^{p_Z} \\
\Spec k\ar[r]^-{(\gamma^*)^{-1}}   & \Spec k
}
\end{equation*}
We shorten ``semilinear morphism'' to ``semi-morphism''.
We write  ``$\nu\colon (Y,p_Y)\to (Z,p_Z)$ is a {\em $\gamma$-semi-morphism}''
if $(\gamma,\nu)\colon (Y,p_Y)\to (Z,p_Z)$ is a semi-morphism.
Then by abuse of notation we write just that $\nu\colon Y\to Z$ is a $\gamma$-semi-morphism.
\end{definition}

Note that if we take $\gamma=\id_k$, then a $\id_k$-semi-morphism $(Y,p_Y)\to (Z,p_Z)$
is just a morphism of $k$-schemes.

\begin{lemma}
If $(\gamma,\nu)\colon (Y,p_Y)\to (Z,p_Z)$ is a semi-morphism of nonempty $k$-schemes,
then the morphism of schemes $\nu\colon Y\to Z$ uniquely determines $\gamma$.
\end{lemma}

\begin{proof}
We may and shall assume that $Y$ and $Z$ are affine, $Y=\Spec R_Y$, $Z=\Spec R_Z$.
Then we have a commutative diagram
\begin{equation}\label{e:RY-RZ}
\begin{aligned}
\xymatrix{
R_Y      & R_Z\ar[l]_(.4){\nu^*}\\
k\ar[u]  & k\ar[u]\ar[l]_(.4){\gamma^{-1}}
}
\end{aligned}
\end{equation}
Since $k$ is a field, the vertical arrows are monomorphisms,
and therefore, the homomorphism of rings $\nu^*$
uniquely determines the automorphism $\gamma^{-1}$.
 \end{proof}

\begin{subsec}
We define an action of a semi-morphism $(\gamma,\nu)\colon (Y,p_Y)\to (Z,p_Z)$ on  $k$-points.
If $y\colon\Spec k\to Y$ is a $k$-point of $(Y,p_Y)$, we set
\begin{equation}\label{e:gamma-semi-points}
(\gamma, \nu)(y)=\nu\circ y\circ\gamma^*\colon \Spec k\to Z,
\end{equation}
which is a $k$-point of $(Z,p_Z)$.
This formula is compatible with the usual formula for the action of a $k$-morphism on $k$-points.
By abuse of notation we  write $\nu(y)$ instead of $(\gamma,\nu)(y)$.

If $(\beta,\mu)\colon (Z,p_Z)\to (W,p_W)$ is a semi-morphism of $k$-schemes, we set
\[ (\beta,\mu)\circ(\gamma,\nu)=(\beta\gamma,\mu\circ\nu).\]
Then clearly $(\beta,\mu)\circ(\gamma,\nu)$ is a semi-morphism, and for every $k$-point $y\in Y(k)$
we have
\begin{equation}\label{e:mu-nu}
(\mu\circ\nu)(y)=\mu(\nu(y)).
\end{equation}
\end{subsec}

\begin{definition}
By a {\em $\gamma$-semi-isomorphism $\nu\colon (Y,p_Y)\to (Z,p_Z)$},
where $(Y,p_Y)$ and $(Z,p_Z)$ are two $k$-schemes,
we mean a $\gamma$-semi-morphism $\nu\colon (Y,p_Y)\to (Z,p_Z)$
for which the morphisms of schemes $\nu\colon Y\to Z$ is an isomorphism.
By a {\em $\gamma$-semi-automorphism} of a $k$-scheme $(Y,p_Y)$
we mean a $\gamma$-semi-isomorphism $\mu \colon (Y,p_Y)\to (Y,p_Y)$.
\end{definition}

\begin{subsec}
Let us fix $\gamma\in\Aut(k)$.
Assume we have two $k$-schemes $(Y,p_Y)$ and $(Z,p_Z)$.
Let $\nu\colon Y\to Z$ be a morphism of schemes.
The diagram with commutative left-hand triangle
\begin{equation}\label{e:semi-morphism-bis}
\begin{aligned}
\xymatrix@R=30pt@C=40pt{
Y\ar[r]^\nu \ar[d]_{p_Y}\ar[rd]^{\gamma_*p_Y}   & Z\ar[d]^{p_Z} \\
\Spec k\ar[r]^-{(\gamma^*)^{-1}}                   & \Spec k
}
\end{aligned}
\end{equation}
shows that
\begin{equation*}
(\gamma,\nu)\colon (Y,p_Y)\to (Z,p_Z)
\end{equation*}
is a semi-morphism, that is, the rectangle commutes,
if and only if the right-hand triangle commutes, that is, if and only if
\begin{equation}\label{e:nu-natural}
(\id_k,\nu)\colon \gamma_*(Y,p_Y)\to (Z,p_Z)
\end{equation}
 is a semi-morphism.
In other words,
$\nu\colon (Y,p_Y)\to (Z,p_Z)$ is a $\gamma$-semi-morphism
if and only if
$\nu\colon \gamma_*(Y,p_Y)\to (Z,p_Z)$ is a $k$-morphism.
For brevity we write
\begin{equation}\label{e:nu-natural-bis}
\nu_\natural\colon \gamma_*Y\to Z
\end{equation}
for the $k$-morphism \eqref{e:nu-natural};
then the $k$-morphism $\nu_\natural$ acts on $k$-points as follows:
\begin{equation}\label{e:action-natural}
(y'\colon\Spec k\to\gamma_*Y)\ \longmapsto\ (\nu\circ y'\colon\Spec k\to Z).
\end{equation}
\end{subsec}

\begin{example}\label{ex:gamma-ast}
Let $(Y,p_Y)$ be a $k$-scheme, and let $\gamma\in\Aut(k)$.
Recall that $\gamma_*(Y,p_Y)=(Y,(\gamma^*)^{-1}\circ p_Y)$.
The commutative diagram
\[
\xymatrix@R=20pt@C=50pt{
Y\ar[r]^{\id_Y} \ar[dd]_-{p_Y}        &Y\ar[d]^-{p_Y} \\
                                 &\Spec k\ar[d]^{(\gamma^*)^{-1}} \\
\Spec k\ar[r]^{(\gamma^*)^{-1}}   &\Spec k
}
\]
shows that  $(\gamma,\id_Y)\colon Y\to\gamma_* Y$
is a $\gamma$-semi-isomorphism.
We denote this $\gamma$-semi-iso\-mor\-phism by
\[ \gamma_!\colon Y\to\gamma_* Y.\]
Comparing formulas \eqref{e:y-!} and \eqref{e:gamma-semi-points},
we see that the $\gamma$-semi-isomorphism $\gamma_!\colon Y\to\gamma_*Y$
acts on $k$-points as the bijective map
$\gamma_!\colon Y(k)\to(\gamma_* Y)(k)$ defined by formula \eqref{e:y-!}.
\end{example}

\begin{subsec}
Let $\nu\colon (Y,p_Y)\to(Z,p_Z)$ be a $\gamma$-semi-morphism.
Then the  diagram \eqref{e:semi-morphism-bis}  commutes, and hence,
\begin{equation*}
\nu=\nu_\natural\circ \gamma_!=\, (\id_k,\nu)\hs\circ\hs (\gamma,\id_Y)\colon
\ Y\labelto{\gamma_!} \gamma_* Y\labelto{\nu_\natural} Z,
\end{equation*}
where $\gamma_!$ is a $\gamma$-semi-isomorphism and $\nu_\natural$ is a $k$-morphism (an $\id_k$-semi-morphism).
It follows that
\begin{equation}\label{e:nat-y'}
\nu(y)=\nu_\natural(\gamma_!(y))\quad \text{for } y\in Y(k)
\end{equation}
(this can be also seen by comparing formulas \eqref{e:y-!},
\eqref{e:gamma-semi-points}, and \eqref{e:action-natural}\hs).
\end{subsec}

\begin{example}\label{ex:b-ch}
Let $Y_0$ be a $k_0$-scheme, where $k_0$ is a subfield of $k$.
For simplicity, we assume that $Y_0$ is affine, that is,
$Y_0=\Spec R_0$\hs, where $R_0$ is a $k_0$-algebra.
We set
\[ Y=Y_0\times_{k_0} k:=Y_0\times_{\Spec k_0} \Spec k.\]
Then
\[ Y=\Spec R,\quad\text{where } R=R_0\otimes_{k_0} k.\]
Let $i\colon R_0\into R$ denote the canonical embedding.  Consider the sets of $k$-points
\[Y_0(k)=\Hom_{k_0} (R_0, k)\quad\text{and}\quad Y(k)=\Hom_k(R,k).\]
We have a canonical morphism of schemes $i^*\colon Y\to Y_0$ inducing a canonical map
\[Y(k)\to Y_0(k), \quad (y\colon \Spec k\to Y)\,\longmapsto \, (y_0=i^*\circ y\colon \Spec k\to Y\to Y_0),\]
which in the language of rings can be written as
\begin{equation}\label{e:R-R0}
 Y(k)\to Y_0(k)\colon\ (y\colon R\to k)\, \longmapsto\, (y_0=y|_{R_0}\colon R_0\to k).
 \end{equation}
The map \eqref{e:R-R0} is bijective; the inverse map is given by
\[Y_0(k)\to Y(k)\colon\ (y_0\colon R_0\to k)\,\longmapsto\,
(r_0\otimes\lambda\mapsto y_0(r_0)\cdot\lambda\in k)\text{ for }r_0\in R_0,\,\lambda\in k.\]

Let $\gamma\in\Aut(k/k_0)$,  that  is, $\gamma$ is an automorphism of $k$ that fixes all elements of $k_0$.
Consider
\[\mu_\gamma=\id_{Y_0}\times (\gamma^*)^{-1}\colon\ Y\to Y.\]
It follows from the construction of $\mu_\gamma$ that the following diagram commutes:
\begin{equation}\label{e:diagram-mu-g}
\begin{aligned}
\xymatrix@R=25pt@C=40pt{  Y_0\ar[r]^{\id}   &Y_0 \\
Y\ar[r]^{\mu_\gamma} \ar[d]_{p_Y}\ar[u]^{i^*}\ar[rd]^{\gamma_*p_Y}          & Y\ar[d]^{p_Y}\ar[u]_{i^*} \\
\Spec k\ar[r]_-{(\gamma^*)^{-1}}   & \Spec k\ar@/^0.6pc/@{..>}[lu]^(0.6){y'} \ar@/^0.55pc/@{..>}[u]^(0.5){y}
}
\end{aligned}
\end{equation}
We see that $\mu_\gamma$ is a $\gamma$-semi-automorphism of $Y$,
and it induces an isomorphism of $k$-schemes
\[(\mu_\gamma)_\natural\colon \gamma_*Y\isoto Y;\]
see formula \eqref{e:nu-natural-bis}.
By formula \eqref{e:action-natural}, $(\mu_\gamma)_\natural$
takes a $k$-point $y'$ of $\gamma^*Y$ to the $k$-point  $y=\mu_\gamma\circ y'$ of $Y$.
We see from the diagram \eqref{e:diagram-mu-g} that
\[i^*\circ y=i^*\circ \mu_\gamma\circ y'=i^*\circ y'.\]
This means that  the isomorphism of $k$-schemes  $(\mu_\gamma)_\natural$
is compatible with the bijections
\[Y(k)\to Y_0(k)\quad\text{and}\quad (\gamma_*Y)(k)\to Y_0(k).\]
We identify the $k$-scheme  $\gamma_* Y$ with  $Y$ via $(\mu_\gamma)_\natural$\hs.
Then
\[\mu_\gamma=\gamma_!\colon Y\to\gamma_*Y=Y.\]
Similar assertions are true when the $k_0$-scheme $Y_0$ is not assumed to be affine.

If $\beta,\gamma\in\Aut(k/k_0)$, then clearly
\begin{equation}\label{e:b-g}
\mu_{\beta\gamma}=\mu_\beta\circ\mu_\gamma\hs.
\end{equation}
By \eqref{e:mu-nu} we obtain that for every $y\in Y(k)$ we have
\begin{equation}\label{e:beta-gamma}
\mu_{\beta\gamma}(y)=\mu_\beta(\mu_\gamma(y)).
\end{equation}
Thus the group $\Aut(k/k_0)$ acts on the set $Y(k)$.
\end{example}

Let $Y$ be an affine $k$-variety, $Y=\Spec R_Y$;
then $R_Y$ is the ring of regular functions on $Y$.
If $f\in R_Y$, then for any $y\in Y(k)$ the value $f(y)\in k$ is defined.

\begin{lemma}\label{l:nu-*-affine}
Let $\nu\colon (Y,p_Y)\to (Z,p_Z)$ be a $\gamma$-semi-isomorphism of {\em affine} $k$-vari\-eties,
where $\gamma\colon k\to k$ is an automorphism of $k$.
Let $Y=\Spec\, R_Y$\hs, $Z=\Spec\,R_Z\hs$,
and let $\nu^*\colon R_Z\to R_Y$ denote the morphism of rings corresponding to $\nu$.
Let $f_Z\in R_Z$. Then
\begin{equation}\label{e:f_Z}
f_Z(\nu(y))=\gamma(\hs(\nu^*\!f_Z)(y)\hs)\quad \text{for all }y\in Y(k).
\end{equation}
\end{lemma}

\begin{proof}
The assumption that $\nu\colon Y\to Z$ is a $\gamma$-semi-morphism
means that the diagram  \eqref{e:RY-RZ} commutes.
A $k$-point $y\in Y(k)$ corresponds to a homomorphism of $k$-algebras
$\vphi_y\colon R_Y\to k$, and
the following diagram commutes:
\[
\xymatrix{
R_Y\ar[d]_{\vphi_y}       & R_Z\ar[d]^{\vphi_{\nu(y)}}\ar[l]_(0.4){\nu^*}  \\
k                    &k\ar[l]_(0.4){\gamma^{-1}}
}
\]
hence $\vphi_{\nu(y)}=\gamma\circ\vphi_y\circ\nu^*$.
We set $f_Y=\nu^*\!f_Z\in R_Y$; then $f_Y(y)=\vphi_y(f_Y)$,
and \eqref{e:f_Z} means that
\[(\gamma\circ\vphi_y\circ\nu^*)(f_Z)=\gamma(\vphi_y(\nu^* f_Z)),\]
which is obvious.
 \end{proof}

Now let $\nu\colon (Y,p_Y)\to (Z,p_Z)$ be a $\gamma$-semi-isomorphism of {\em irreducible} $k$-varieties,
where $\gamma\colon k\to k$ is an automorphism of $k$.
Then the isomorphism of schemes $\nu\colon Y\to Z$ induces an isomorphism of the fields of rational functions
\[\nu_*\colon K(Y)\to K(Z),\quad f\mapsto \nu_*f.\]
For any $f\in K(Y)$ and $y\in Y(k)$, the   value $f(y)\in \kinf$ of $f$ at $y$ is defined, where
we write $f(y)=\infty$ if $f$ is not regular at $y$.

\begin{corollary}\label{l:nu-*}
Let $\nu\colon (Y,p_Y)\to (Z,p_Z)$ be a $\gamma$-semi-isomorphism of irreducible $k$-varieties,
where $\gamma\colon k\to k$ is an automorphism of $k$.
With the above notation  we have
\begin{equation*}
(\nu_*f_Y)(z)=\gamma(f_Y(\nu^{-1}(z)))\quad \text{for all } f_Y\in K(Y),\ z\in Z(k).
\end{equation*}
\end{corollary}

\begin{proof}
We consider the isomorphism
\[\nu^*=\nu_*^{-1}\colon K(Z)\to K(Y),\quad f_Z\mapsto\nu^*\!f_Z,\quad\text{where }f_Z\in K(Z).\]
Set $f_Z=\nu_*f_Y$; then $f_Y=\nu^* f_Z$.
We must prove that \eqref{e:f_Z} holds.
We may and shall assume that $Y$ and $Z$ are affine varieties,
$Y=\Spec\, R_Y$\hs, $Z=\Spec\,R_Z$\hs, $f_Z\in R_Z$\hs,
and that the morphism $\nu$ corresponds to a homomorphism of rings
$\nu^*\colon R_Z\to R_Y$\hs.
Now the corollary follows from Lemma \ref{l:nu-*-affine}.
 \end{proof}

\begin{subsec}\label{r:classical-language} ({\em Classical language})
In this subsection we describe the variety   $\gamma_* Y$
and the map $\gamma_!\colon Y(k)\to(\gamma_*Y)(k)$ in the language of classical algebraic geometry.
First, consider the $n$-dimensional affine space  $\Ank$; then $\Ank(k)=k^n$.
Let $k_0$ be the prime subfield of $k$, that is, the subfield generated by 1;
then $\Ank=\Anknull\times_{k_0}k$.
Let $\gamma\in\Aut(k)=\Aut(k/k_0)$; then $\gamma$ induces
a $\gamma$-semi-automorphism $\mu_\gamma\colon\Ank\to\Ank$;
see Example \ref{ex:b-ch}.
As in Example \ref{ex:b-ch}, we identify $\gamma_*\Ank$ with $\Ank$ using the $k$-isomorphism
\[(\mu_\gamma)_\natural\colon \gamma_*\Ank\isoto\Ank\hs;\]
then
\[\mu_\gamma=\gamma_!\colon \Ank\to\gamma_*\Ank=\Ank\hs,\]
and $\mu_\gamma(x)=\gamma_!(x)$ for $x\in \Ank(k)$.
For $i=1,\dots,n,$ let $f_i$ denote the $i$-th coordinate function on $\Ank$, which is a regular function.
Since $f_i$ comes from a regular function on $\Anknull$, we have $(\mu_\gamma)_*f_i=f_i$,
and by Lemma \ref{l:nu-*-affine} we have
\[ f_i(\mu_\gamma(x))=\gamma(f_i(x))\quad\text{for } x\in \Ank(k)=k^n.\]
If we write $x=(x_i)_{i=1}^n\in k^n=\Ank(k)$, where $x_i=f_i(x)\in k$, then
\[ \mu_\gamma(x)=\gamma(x_i)_{i=1}^n\hs.\]

Similarly, let $\PP_k^n$ denote the $n$-dimensional projective space over $k$; then $\PP_k^n=\PP_{k_0}^n\times_{k_0} k$.
We denote by $x=(x_0:x_1:\dots:x_n)\in \PP^n_k(k)$ the $k$-point with homogeneous coordinates $x_0,x_1,\dots,x_n$.
Then for $\gamma\in\Aut(k/k_0)$, the $\gamma$-semi-automorphism $\mu_\gamma$ of $\PP^n_k$ takes
$x$ to $\gamma_!(x)=(\gamma(x_0):\gamma(x_1):\dots:\gamma(x_n)\hs)$.

Now let $Y\subset \Ank$ be an affine variety (a closed subvariety of $\Ank$).
Let $\iota\colon Y\into \Ank$ denote the inclusion morphism; then $\gamma$ induces a $k$-morphism
\[ \gamma_*\iota\colon\gamma_*Y\to\gamma_*\Ank=\Ank\hs.\]
From the commutative diagram
\[
\xymatrix{
Y(k)\ar[r]^-{\gamma_!}\ar[d]_\iota    & (\gamma_*Y)(k)\ar[d]^{\gamma_*\iota}  \\
\Ank(k)   \ar[r]^-{\gamma_!}       &\Ank(k)
}
\]
we see that
\[ (\gamma_*\iota)(y')=\gamma_!(\iota(\gamma_!^{-1}(y')\hs)\hs) \quad\text{for } y'\in (\gamma_*Y)(k),\]

Now we assume that $k$ is algebraically closed.
As usual in classical algebraic geometry,
we identify an affine variety $Y\subset \Ank$ with the algebraic set
\[ Y(k)\subset \Ank(k)=k^n.\]
Furthermore, we identify $\gamma_*Y$ with the algebraic set
\[(\gamma_*\iota)(\gamma_*Y(k))=\gamma_!(Y(k))\subset k^n\hs.\]
We see that
\[\gamma_*Y=\{\hs\gamma(y_i)_{i=1}^n\ |\ (y_i)_{i=1}^n\in Y\hs\}, \]
and that the map
\[ \gamma_!\colon Y\to\gamma_*Y\]
sends a point  $y$ with coordinates  $(y_i)_{i=1}^n$ to the point with coordinates  $\gamma(y_i)_{i=1}^n$.
If $Y\subset k^n$ is defined by a family of polynomials $(P_\alpha)_{\alpha\in A}$,
then $\gamma_*Y\subset k^n$ is defined by the family $(\gamma(P_\alpha))_{\alpha\in A}$\hs,
where $\gamma(P_\alpha)$ is the polynomial obtained from $P_\alpha$
by acting by $\gamma$ on the coefficients.

One can describe similarly $\gamma_*Y$ when $Y\subset\PP^n_k$ is a projective or quasi-projective variety.
Namely,
\begin{equation*}
\gamma_*Y=\{\gamma_!(y)\ |\ y\in Y\}\subset \PP^n_k\hs.
\end{equation*}
\end{subsec}

\section{Semi-morphisms of $G$-varieties}
\label{s:G}

\begin{subsec}\label{ss:semi-auto-of-G}
In this section $k$ is an algebraically closed field, and $Y$ is a $k$-variety,
that is, a reduced separated scheme of finite type over $k$.

Let $G$ be an algebraic group over $k$ (we write also ``an algebraic $k$-group''), that is,
a smooth group scheme of finite type over $k$.
Let $(Y,\theta)$ be a $G$-$k$-variety, that is, a $k$-variety $Y$ together with an action
\[\theta\colon G\times_k Y\to Y\]
of $G$ on $Y$.
If $g\in G(k)$ and $y\in Y(k)$, we write just $g\uY y$ or $g\cdot y$ for $\theta(g,y)\in Y(k)$.
\end{subsec}

\begin{definition}[\rm cf. {\cite[(1.2)]{FSS}}\hs]
\label{d:tau-gamma}
Let $\gamma\in \Aut(k)$. A {\em $\gamma$-semi-automorphism of an algebraic $k$-group $G$}
 is a $\gamma$-semi-automorphism of $k$-schemes $\tau\colon G\to G$
such that the corresponding isomorphism of $k$-varieties
$\tau_\natural\colon\gamma_* G\to G$, see  \eqref{e:nu-natural-bis},
is an isomorphism of algebraic $k$-groups.
This condition is the same as to require that certain diagrams
containing $\tau$ commute; see \cite[Section 1.2]{Borovoi-Duke}.
\end{definition}

Let $H\subset G$ be an algebraic $k$-subgroup.
By  Definition \ref{d:tau-gamma} we have
$$\tau(H(k))=\tau_\natural(\hs(\gamma_*H)(k)\hs),$$
where $\tau_\natural\colon \gamma_* G\to G$ is an isomorphism of algebraic $k$-groups.
It follows that  $\tau(H(k))$ is the set of $k$-points of a $k$-subgroup of $G$, which we denote by $\tau(H)$.
By definition,  $\tau(H)(k)=\tau(H(k))$.

\begin{subsec}\label{ss:G0}
Let $k_0\subset k$ be a subfield.
We shall always assume that we are given a $k_0$-model $(G_0,\vk_G)$ of $G$; see the Introduction.
Let $\gamma\in\Aut(k/k_0)$, that is, $\gamma$ be an automorphism of $k$ fixing all elements of $k_0$.
Then we have  $\gamma$-semi-automorphisms
\begin{equation}\label{e:sigma-gamma}
\sigma^\circ_\gamma=\id_{G_0}\times(\gamma^*)^{-1}\colon G_{0,k}\to G_{0,k} \quad \text{and}\quad \sigma_\gamma\colon G\to G,
\end{equation}
compare Example \ref{ex:b-ch},
where $G_{0,k}=G_0\times_{k_0} k$ and we obtain $\sigma_\gamma$ from $\sigma_\gamma^\circ$ via the $k$-isomorphism $\vk_G\colon G_{0,k}\isoto G$.
If $\beta,\gamma \in\Aut(k/k_0)$, then by \eqref{e:b-g} we have
\begin{equation*}
\sigma_{\beta\gamma}=\sigma_\beta\circ\sigma_\gamma\hs,
\end{equation*}
and by \eqref{e:beta-gamma} we have
\begin{equation}\label{e:beta-gamma-sigma}
\sigma_{\beta\gamma}(g)=\sigma_\beta(\sigma_\gamma(g)\hs)\quad \text{for every } g\in G(k).
\end{equation}

If $\alpha$ is any $k$-automorphism of $G$, then
\[\tau:=\alpha\circ\sigma_\gamma\colon G\to G\]
is a $\gamma$-semi-automorphism of $G$, and all $\gamma$-semi-automorphisms of $G$ (for given  $\gamma$)
can be obtained in this way.

If we assume that $G=G_0\times_{k_0}k$, where $G_0$ is an affine algebraic group contained in $\GL_{n,k_0}$
and defined by polynomials $(P_\alpha)$ with coefficients in $k_0$,
then by Subsection \ref{r:classical-language} we have $\sigma_\gamma(g)=\gamma(g)$ for every $g\in G_0(k)$,
and clearly
\begin{equation}\label{e:g1g2}
\sigma_\gamma(g_1\hs g_2)=\sigma_\gamma(g_1)\cdot\sigma_\gamma(g_2)\quad\text{for every } g_1,\hs g_2\in G_0(k).
\end{equation}
In general, for any $k_0$-model $G_0$ of $G$, since the multiplication law in $G$ is ``defined over $k_0$'',
 one can easily show that again \eqref{e:g1g2} holds.
By \eqref{e:beta-gamma-sigma} the group $\Aut(k/k_0)$ acts on the set $G_0(k)=G(k)$, and by \eqref{e:g1g2}
this action preserves the group structure in $G(k)$.
\end{subsec}

\begin{subsec}
Let $G_0$ be as in \ref{ss:G0} and let $\sigma_\gamma$ be as in \eqref{e:sigma-gamma}.
If we have a  $G$-$k$-variety $(Y,\theta)$ and
a $G_0$-equivariant $k_0$-model $(Y_0,\theta_0)$ of $(Y,\theta)$ (see the Introduction),
then for any $\gamma\in \Aut(k/k_0)$ we obtain a $\gamma$-semi-auto\-morphism $\mu_\gamma\colon Y\to Y$,
compare  Example \ref{ex:b-ch} and Subsection \ref{ss:G0}.
Since $\theta$ is ``defined over $k_0$'', it is easy to see that
\[\mu_\gamma(g\cdot y)=\sigma_\gamma(g)\cdot\mu_\gamma(y)\quad\text{for all }g\in G(k),\ y\in Y(k).\]
\end{subsec}

\begin{definition}\label{d:g-semi-morphism}
Let $G$ be an algebraic $k$-group, and
let $(Y,\theta_Y)$ and $(Z,\theta_Z)$ be two $G$-$k$-varieties.
Let $\gamma\in\Aut(k)$, and let $\tau\colon G\to G$
be a $\gamma$-semi-automorphism of $G$.
A {\em $\tau$-equivariant $\gamma$-semi-morphism}
\[\nu\colon (Y,\theta_Y)\to (Z,\theta_Z)\]
is a $\gamma$-semi-morphism $\nu\colon Y\to Z$ such that the following diagram commutes:
\begin{equation}\label{e:g-semi-morphism}
\begin{aligned}
\xymatrix{
G\times_k Y\ar[r]^-{\theta_Y}\ar[d]_{\tau\times\nu}   & Y \ar[d]^\nu\\
G\times_k Z\ar[r]^-{\theta_Z}                                     &Z
}
\end{aligned}
\end{equation}
where we write $\tau\times \nu$ for the product of $\tau$ and $\nu$
over the automorphism $(\gamma^*)^{-1}$ of $\Spec k$.
\end{definition}

Since $k$ is algebraically closed, $G$ is smooth (reduced),  $Y$ and $Z$ are reduced, we see that
the diagram \eqref{e:g-semi-morphism} commutes if and only if
\[\nu(g\cdot y)=\tau(g)\cdot\nu(y)\quad\text{for all }g\in G(k),\ y\in Y(k).\]

\begin{construction}\label{ss:*Y*Y}
Let $G$ be an algebraic $k$-group, and let $(Y,\theta_Y)$ be a $G$-$k$-variety.
The group $\gamma_*G$ naturally acts on $\gamma_* Y$:
the action $\theta\colon G\times_k Y\to Y$ induces an action
\begin{equation}\label{e:ast-ast}
\gamma_*\theta\colon \gamma_* G\times_k\gamma_* Y\to\gamma_* Y.
\end{equation}
By definition, a $\gamma$-semi-automorphism $\tau$ of $G$ defines an isomorphism of algebraic $k$-groups
\[ \tau_\natural\colon\gamma_* G\to G.\]
We identify $G$ and $\gamma_* G$ via $\tau_\natural$ and obtain from \eqref{e:ast-ast} an action
\[ \tau^*\gamma_*\theta\colon\hs  G\hs\times_k\hs \gamma_*Y\to\gamma_*Y,\ \,
            (g,y')\mapsto(\gamma_*\theta)(\tau_\natural^{-1}(g),y')
                      \ \, \text{for }g\in G(k),\ y'\in (\gamma_*Y)(k).\]
By abuse of notation, we write 
 $\gamma_*Y$ for the $G$-$k$-variety $(\gamma_*Y,\tau^*\gamma_*\theta)$.
We write
\[ g\uT y'\quad\text{for } (\tau^*\gamma_*\theta)(g,y')=(\gamma_*\theta)(\tau_\natural^{-1}(g),y'),
                                                    \quad \text{where } g\in G(k),\ y'\in(\gamma_*Y)(k).\]
By formula \eqref{e:nat-y'} we have
$\tau(g)=\tau_\natural(\gamma_!(g)),$
and hence,
\begin{equation}\label{e:*-tau}
\begin{aligned}
g\uT y'=(\gamma_*\theta)(\tau_\natural^{-1}(g),y')
&=\gamma_!(\theta(\gamma_!^{-1}(\tau_\natural^{-1}(g)),\gamma_!^{-1}(y')))\\
&=\gamma_!(\hs\theta(\tau^{-1}(g),\gamma_!^{-1}(y')))=\gamma_!(\tau^{-1}(g)\uY\gamma_!^{-1}(y')).
\end{aligned}
\end{equation}
\end{construction}

\begin{lemma}\label{l:*-Stab}
Let $G$ be an algebraic $k$-group, and let $(Y,\theta)$ be a $G$-$k$-variety.
Let  $\gamma\in \Aut(k)$, and let $\tau\colon G\to G$ be a $\gamma$-semi-automorphism of $G$.
Let $y^{(0)}\in Y(k)$ be a $k$-point, and write $H=\Stab_G(y^{(0)})$.
Consider the action
\[\tau^*\gamma_*\theta\colon\  G\times_k\gamma_*Y\to \gamma_*Y.\]
Then the stabilizer in $G(k)$ of the point $\gamma_!(y^{(0)})\in (\gamma_*Y)(k)$
under the action  $\tau^*\gamma_*\theta$ is $\tau(H(k))=\tau(H)(k)$.
\end{lemma}

\begin{proof}
By formula \eqref{e:*-tau} we have
\[g\uT\gamma_!(y^{(0)})=\gamma_!(\tau^{-1}(g)\cdot y^{(0)}).\]
Since the stabilizer in $G(k)$ of $y^{(0)}\in Y(k)$ is $H(k)$, the lemma follows.
 \end{proof}

Note that the $\gamma$-semi-morphism  $\nu$ in Definition
\ref{d:g-semi-morphism} defines a $k$-morphism
\[ \nu_\natural\colon \gamma_*Y\to Z;\]
see \eqref{e:nu-natural-bis}.

\begin{lemma}\label{l:equi}
Let $\gamma\in\Aut(k)$ and let $\tau\colon G\to G$ be a $\gamma$-semi-automorphism of $G$.
Let $(Y,\theta_Y)$ and $(Z,\theta_Z)$ be two $G$-$k$-varieties.
A morphism of schemes $\nu\colon Y\to Z$
is a $\tau$-equivariant $\gamma$-semi-morphism
if and only if $\nu_\natural \colon\gamma_*Y\to Z$
is a $G$-equivariant morphism of $k$-varieties
(where we write $\gamma_*Y$ for $(\gamma_*Y,\tau^*\gamma_*\theta_Y)$\hs).
\end{lemma}

\begin{proof}
By \eqref{e:semi-morphism-bis} the morphism of schemes $\nu$
is a $\gamma$-semi-morphism $Y\to Z$ if and only if it is a $k$-morphism $\gamma_* Y\to Z$.

Let $g\in G(k)$, $y'\in (\gamma_*Y)(k)$.
Using formula \eqref{e:*-tau} we obtain
\begin{equation}\label{e:nu-natural-22}
\nu_\natural(g\uT y')=\nu(\gamma_!^{-1}(g\uT y'))=\nu(\hs\tau^{-1}(g)\uY\gamma_!^{-1}(y')\hs).
\end{equation}
We have also
\begin{equation}\label{e:g-ast-Z}
g\uZ\nu(\gamma_!^{-1}(y'))=g\uZ\nu_\natural(y').
\end{equation}

If $\nu$ is $\tau$-equivariant, then
\begin{equation*}
\nu(\tau^{-1}(g) \uY \gamma_!^{-1}(y'))=g \uZ\nu(\gamma_!^{-1}(y')),
\end{equation*}
and we obtain using \eqref{e:g-ast-Z} that
\begin{equation}\label{e:nu-tau-g-25'}
\nu(\tau^{-1}(g) \uY \gamma_!^{-1}(y'))=g\uZ\nu_\natural(y').
\end{equation}
From \eqref{e:nu-natural-22} and \eqref{e:nu-tau-g-25'} we obtain that
\[\nu_\natural(g\uT y')=\nu(\tau^{-1}(g) \uY \gamma_!^{-1}(y'))=g\uZ\nu_\natural(y')\quad
            \text{for all } g\in G(k),\ y'\in (\gamma_*Y)(k),\]
hence $\nu_\natural$ is $G$-equivariant.

Conversely, if $\nu_\natural$ is $G$-equivariant,
then
$\nu_\natural(g\uT y')=g\uZ\nu_\natural(y')$,
and we obtain using \eqref{e:nu-natural-22} and \eqref{e:g-ast-Z} that
\begin{equation}\label{e:nu-tau-g-23}
\nu(\tau^{-1}(g)\uY\gamma_!^{-1}(y'))=\nu_\natural(g\uT y')
              =g\uZ\nu_\natural(y')=g\uZ \nu(\gamma_!^{-1}(y')).
\end{equation}
Set $g'=\tau^{-1}(g)\in G(k)$,  $y=\gamma_!^{-1}(y')\in Y(k)$.
Then we obtain from \eqref{e:nu-tau-g-23} that
\[\nu(g'\uY y)=\tau(g')\uZ \nu(y)\quad\text{for all } g'\in G(k),\ y\in Y(k).\]
Thus $\nu$ is $\tau$-equivariant.
 \end{proof}

\begin{corollary}\label{c:tau-equi-semi}
Let $\gamma$ be an automorphism of $k$, and let $\tau\colon G\to G$ be a $\gamma$-semi-auto\-mor\-phism of $G$.
Let $(Y,\theta)$ be a $G$-$k$-variety.
There exists a $\tau$-equivariant $\gamma$-semi-automorphism $\mu\colon Y\to Y$
if and only if the $G$-$k$-variety $(\gamma_*Y,\tau^*\gamma_*\theta)$ is isomorphic to $(Y,\theta)$.
\end{corollary}

\begin{proof}
We take $Z=Y$ in Lemma \ref{l:equi}.
 \end{proof}

\section{Quotients}
\label{s:quotients}

\begin{subsec}
Let $k$ be a field (not necessarily algebraically closed).
By an {\em algebraic scheme} over $k$ we mean a scheme of finite type over $k$.
By an {\em algebraic group scheme} over $k$ we mean a group scheme over $k$ whose
underlying scheme is of finite type over $k$.

 Let $H$ be an algebraic group subscheme of an algebraic group $k$-scheme $G$.
A {\em quotient} of $G$ by $H$ is an algebraic scheme $Y$ over $k$ equipped with an action
$\theta\colon G\times_k Y\to Y$ and a point $y^{(0)}\in Y(k)$ fixed by $H$ satisfying certain conditions
(a) and (b); see Milne \cite[Definition 5.20]{Milne}.

By  \cite[Theorem 5.28]{Milne} there does exist a quotient of $G$ by $H$.
By \cite[Proposition 5.22]{Milne} this quotient $(Y,\theta,y^{(0)})$ has the following universal property:

\begin{property*}[\rm U]
Let $Z$ be a $k$-scheme on which $G$ acts, and let $z^{(0)}\in Z(k)$ be a point fixed by $H$.
Then there exists a unique $G$-equivariant map $Y\to Z$ making the following diagram commute:
\[
\xymatrix{
G\ar[rr]^{g\mapsto g\cdot y^{(0)} } \ar[rrd]_{g\mapsto g\cdot z^{(0)} } & & Y\ar[d] \\
                                                                       &  & Z
}
\]
\end{property*}

Clearly the universal property (U) uniquely determines the quotient up to a unique isomorphism,
so we may take (U) as a definition of the quotient.
\end{subsec}

\begin{subsec}
We return to our settings: $k$ is an {\em algebraically closed} field.
Let $G$ be a {\em linear} algebraic $k$-group (a smooth affine group $k$-scheme of finite type over $k$)
and $H$ be a {\em smooth} algebraic $k$-subgroup of $G$.
Under these assumptions,
a classical construction of Chevalley (see, for instance, \cite[Theorem 4.27]{Milne}\hs) shows that
the quotient $Y$ exists as a quasi-projective variety, see \cite[Theorem 7.18]{Milne}.
Thus  $Y$ is a $k$-variety, and therefore, in the universal property (U) defining $Y$
we may assume that $Z$ is a $k$-variety.
Since $k$ is algebraically closed and $H$ is smooth, the condition ``fixed by $H$''
is equivalent to ``fixed by $H(k)$''.
Thus we arrive to the following definition of Springer:
\end{subsec}

\begin{definition}[\rm cf. Springer {\cite[Section 5.5]{Springer-LAG}}]
\label{d:Springer}
Let $k$ be an algebraically closed  field,
and let $G$ be a linear algebraic $k$-group.
Let $H\subset G$ be a  (smooth)  $k$-subgroup.
A {\em quotient of $G$ by $H$} is a pointed $G$-$k$-variety
$(Y,\ \theta\colon G\times_k Y\to Y,\  y^{(0)}\in Y(k)\,)$
such that $H(k)$ fixes $y^{(0)}$, with the following universal property:

\begin{property*}[\rm U$'$]
 For any pointed $G$-$k$-variety $(Z,\theta_Z,z^{(0)})$
with the $k$-point $z^{(0)}\in Z(k)$ fixed by $H(k)$,
there exists a unique morphism  of pointed $G$-$k$-varieties
$(Y,\theta,y^{(0)})\ \to\ (Z,\theta_Z,z^{(0)})$.
\end{property*}
\end{definition}

\begin{subsec}
For $G$ and $H$ as in Definition \ref{d:Springer}, let $(Y,\theta, y^{(0)})$ be a quotient of $G$ by $H$.
The action of $G$ on $Y$ induces a {\em $G$-$k$-morphism} (a morphism of $G$-$k$-varieties)
\begin{equation}\label{e:lambda-Y}
 G\to Y,\quad g\mapsto g\cdot y^{(0)}\quad \text{for } g\in G,
\end{equation}
where $G$ acts on itself by left translations.

As usual, we write $G/H$ for $Y$ and $g\cdot H$ or $gH$ for $g\cdot y^{(0)}$, where $g\in G(k)$.
In particular, we write $1\cdot H$ for $y^{(0)}$.
The $G$-equivariant morphism $G/H=Y \to Z$ of (U$'$) sends $1\cdot H\in (G/H)(k)$ to $z^{(0)}$,
hence for every $g\in G(k)$ it sends the $k$-point $gH\in(G/H)(k)$  to $g\cdot z^{(0)}\in Z(k)$.
Thus the quotient $G/H$ has the following universal property:

\begin{property*}[\rm U$''$]
For any pointed $G$-$k$-variety $(Z,\theta_Z,z^{(0)})$
such that the $k$-point $z^{(0)}$ is fixed by $H(k)$,
there exists a unique $G$-$k$-morphism
$G/H\to Z$ sending $gH$ to $g\cdot z^{(0)}$ for every $g\in G(k)$.
\end{property*}

By \cite[Definition 5.20(a)]{Milne} the morphism \eqref{e:lambda-Y}  induces an injective map
\[G(k)/H(k)\to (G/H)(k), \quad g\cdot H(k)\mapsto  gH.\]
By \cite[Proposition 5.25]{Milne} the morphism \eqref{e:lambda-Y} is faithfully flat, and therefore,
since $k$ is algebraically closed, we see that the induced map $G(k)/H(k)\to (G/H)(k)$ is surjective.
We conclude that this map is bijective.
Thus any $k$-point of $G/H$ is of the form $gH$, where $g\in G(k)$.
\end{subsec}

\section{Semi-morphisms of homogeneous spaces}
\label{s:homog}

Let $k$ be an algebraically closed field.
All algebraic $k$-groups are assumed to be linear and smooth,
and all $k$-subgroups are assumed to be smooth.

\begin{lemma}[\rm well-known]
\label{l:isom-conj}
Let $G$ be a linear algebraic $k$-group over an algebraical\-ly closed field $k$,
and let $H_1, H_2$ be two $k$-subgroups.
Then $Y_1=G/H_1$ and $Y_2=G/H_2$ are isomorphic
as $G$-$k$-varieties
if and only if the subgroups $H_1$ and $H_2$ are conjugate.
To be more precise, for $a\in G(k)$ the following two assertions are equivalent:
\begin{enumerate}
\item[\rm(i)] There exists an isomorphism of $G$-$k$-varieties $\phi_a\colon G/H_1\to G/H_2$
taking $g\cdot H_1$ to $ga^{-1}\cdot H_2$ for $g\in G(k)$;
\item[\rm(ii)] $H_1=a^{-1} H_2\hs a$.
\end{enumerate}
\end{lemma}

\begin{proof}
(i)$\Rightarrow$(ii).
Clearly
\[\Stab_{G(k)}(1\cdot H_1)=H_1(k)\quad \text{and} \quad\Stab_{G(k)}(a^{-1}\cdot H_2)=a^{-1}\cdot H_2(k)\cdot a.\]
Since $\phi_a(1\cdot H_1)=a^{-1}\cdot H_2$, these stabilizers coincide, whence (ii).

(ii)$\Rightarrow$(i).
Set \,$Y_2=G/H_2$, \,$y_2^{(0)}=1\cdot H_2\in Y_2(k)$, \,$y'=a^{-1}\cdot y_2^{(0)}\in Y_2(k)$;
then $\Stab_{G(k)}(y')=a^{-1} H_2(k)\hs a=H_1(k)$,
 so by the property (U$''$) of the quotient $G/H_1$ there exists a unique morphism
 of $G$-varieties $\phi_a\colon G/H_1\to G/H_2$
 such that
 \[\phi_a(g\cdot H_1)=g\cdot a^{-1}\cdot H_2\quad\text{for } g\in G(k).\]
Similarly, since the stabilizer in $G(k)$ of \, $a\cdot H_1\in (G/H_1)(k)$ \, is $H_2(k)$,
there exists a unique  morphism of $G$-varieties $\psi_a\colon G/H_2\to G/H_1$
such that
 \begin{equation*}
 \psi_a(g\cdot H_2)=g\cdot a\cdot H_1\quad\text{for } g\in G(k).
 \end{equation*}
 Clearly these two morphisms are mutually inverse, hence both $\phi_a$ and $\psi_a$ are isomorphisms.
 \end{proof}

\begin{definition}
Let $G$ be a linear algebraic group over an algebraically closed field $k$.
Let $Y$ be a $G$-$k$-variety.
We denote by $\Aut^G(Y)$ the group of $G$-equivariant $k$-automorphisms of $Y$, that is,
of $k$-automorphisms $\psi\colon Y\to Y$ such that
\[ \psi(g\cdot y)=g\cdot \psi(y) \quad\text{for }g\in G,\ y\in Y.\]
\end{definition}

\begin{corollary}[\rm well-known]
\label{l:k-automorphisms}
Let $G$ be a linear algebraic group over an algebraic\-al\-ly closed field $k$.
Let $Y$ be a homogeneous $G$-$k$-variety, that is,  $Y=G/H$,  where $H$ is a $k$-subgroup of $G$.
Set  $N=\NGH$, the normalizer of $H$ in $G$.
For  $n\in N(k)$ we define a map on $k$-points
\begin{equation*}\label{e:Aut-G-H}
n_*\colon G/H\to G/H,\ \  gH\, \longmapsto\ gHn^{-1}=gn^{-1} H\quad \text{for }g\in G(k).
\end{equation*}
Then
\begin{enumerate}
\item[\rm(i)] The map $n_*$ is induced by some automorphism $\phi_n\in \Aut^G(G/H)$;
\item[\rm(ii)] The map
\begin{equation}\label{e:N(H)-H}
\phi\colon N(k)\to\Aut^G(G/H),\quad n\mapsto \phi_n
\end{equation}
is a homomorphism inducing an isomorphism
\begin{equation*}
 N(k)/H(k)\isoto \Aut^G(G/H).
\end{equation*}
\end{enumerate}
\end{corollary}

\begin{proof}
By assumption $n^{-1} H n=H$, and by Lemma \ref{l:isom-conj}
there exists an isomorphism  $\phi_n\colon G/H\to G/H$ such that $\phi_n(g\cdot H)=gn^{-1}\cdot H$,
which proves (i).

Clearly the map $\phi$  of \eqref{e:N(H)-H} is a homomorphism with kernel $H(k)$.
To prove (ii) it remains to show that $\phi$ is surjective,
which is straightforward.
 \end{proof}

\begin{corollary}\label{c:NGH}
If $\NGH=H$, then $\Aut^G(G/H)=\{1\}$.\qed
\end{corollary}

\begin{subsec}\label{ss:k0-tau}
Let $k$ be an algebraically closed field.
Let $G$ be a  linear algebraic group over $k$.
Let $\gamma\in\Aut(k)$.
Let $\tau\colon G\to G$ be a $\gamma$-semi-automorphism of $G$.

Let $H\subset G$ be a (smooth) $k$-subgroup.
Set $Y=G/H$; then we have a morphism $\theta\colon G\times_k Y\to Y$ defining the action of $G$ on $Y$.
Furthermore, the variety $Y$ has a $k$-point $y^{(0)}=1\cdot H$ such that $\Stab_{G(k)}(y^{(0)})=H(k)$,
and the group of $k$-points $G(k)$ acts on $Y(k)$ transitively.

Consider the variety $\gamma_* Y$, the action
$\gamma_*\theta\colon \gamma_*G\times_k\gamma_*Y\to \gamma_*Y$ of $\gamma_*G$ on $\gamma_*Y$,
and the $k$-point $\gamma_!(y^{(0)})\in (\gamma_*Y)(k)$.
As in Construction \ref{ss:*Y*Y}  we obtain an action
\[\tau^*\gamma_*\theta\colon\  G\times_k\gamma_*Y\to \gamma_*Y.\]
\end{subsec}

\begin{lemma}\label{l:nu-1-2}
Let $k$ be an algebraically closed field, and
let $G$ be a  linear algebraic group over $k$.
Let $\gamma\in\Aut(k)$, and
let $\tau\colon G\to G$ be a $\gamma$-semi-automorphism of $G$.
Let $H\subset G$ be a smooth $k$-subgroup.
Set  $Y=G/H$ and let $\theta\colon G\times_k Y\to Y$ denote the canonical action.
Consider the map on $k$-points
\begin{equation}\label{e:nu-1-2}
 (G/H)(k)\to (G/\tau(H))(k),\quad g\cdot H\mapsto \tau(g)\cdot \tau(H) \quad \text{for }g\in G(k).
\end{equation}
Then the following assertions hold:
\begin{enumerate}
\item[\rm (i)] The pointed $G$-$k$-variety $(\gamma_*Y,\hs\tau^*\gamma_*\theta,\hs \gamma_!(y^{(0)}))$
is isomorphic to $G/\tau(H)$;
\item[\rm (ii)] the map \eqref{e:nu-1-2}  is induced by some
$\gamma$-semi-isomorphism $\nu\colon G/H\to G/\tau(H)$.
\end{enumerate}
\end{lemma}

\begin{proof}
Consider the pointed $G$-$k$-variety $(\gamma_*Y,\hs\tau^*\gamma_*\theta,\hs \gamma_!(y^{(0)}))$.
By Lemma \ref{l:*-Stab}, the subgroup $\tau(H(k))=\tau(H)(k)$ of $G(k)$ fixes $\gamma_!(y^{(0)})$.

Now let $(Z,\theta_Z, z^{(0)})$ be a pointed $G$-$k$-variety such that $\tau(H(k))$ fixes $z^{(0)}$.
Consider the pointed $G$-$k$-variety
$$(\hs(\gamma^{-1})_*Z,\hs (\tau^{-1})^*(\gamma^{-1})_*\theta_Z,\hs (\gamma^{-1})_!\hs(z^{(0)})\hs),$$
where the action
\begin{equation}\label{e:tau-1-Z}
(\tau^{-1})^*(\gamma^{-1})_*\theta_Z\colon \ G\times_k (\gamma^{-1})_*Z\ \to\  (\gamma^{-1})_*Z
\end{equation}
is defined as in Construction \ref{ss:*Y*Y},
but  for the pair $(\gamma^{-1},\tau^{-1})$ instead of $(\gamma,\tau)$.
By Lemma \ref{l:*-Stab} applied to $(\gamma^{-1},\tau^{-1})$, the group $H(k)$
fixes $(\gamma^{-1})_!\hs(z^{(0)})\in (\gamma^{-1})_*Z$.
For any morphism of pointed $G$-$k$-varieties
\begin{equation}\label{e:nu-nu}
\vk\colon (Y,\theta, y^{(0)})\to (\hs(\gamma^{-1})_*Z,\, (\tau^{-1})^*(\gamma^{-1})_*\theta_Z,\, (\gamma^{-1})_!\hs(z^{(0)})\hs)
\end{equation}
we obtain a morphism of pointed $G$-$k$-varieties
\begin{equation}\label{e:gamma*nu-nu}
\gamma_*\vk\colon (\gamma_*Y,\hs\tau^*\gamma_*\theta,\hs \gamma_!(y^{(0)})\hs)\to (Z, \theta_Z, z^{(0)}).
\end{equation}
We see that the map $\vk\mapsto \gamma_*\vk$ is a bijection
between the set of morphisms as in \eqref{e:nu-nu}
and the set of morphisms as in \eqref{e:gamma*nu-nu}.
Since $Y=G/H$ and  $H(k)$ fixes $(\gamma^{-1})_!\hs(z^{(0)})$ under the action \eqref{e:tau-1-Z},
we conclude by the universal property (U$'$) for the quotient $Y=G/H$,
that the former set contains exactly one element.
It follows that the latter set contains exactly one element,
that is, the pointed $G$-$k$-variety $(\gamma_*Y,\hs\tau^*\gamma_*\theta,\hs \gamma_!(y^{(0)})\hs)$
has the universal property (U$'$).
This means that, by Definition \ref{d:Springer},
the pointed $G$-$k$-variety $(\gamma_*Y,\hs\tau^*\gamma_*\theta,\hs \gamma_!(y^{(0)})\hs)$
is a quotient of $G$ by $\tau(H)$, which proves (i).
It follows that there exists an isomorphism of $G$-$k$-varieties
\begin{equation}\label{e:lambda-gamma-star}
\lambda\colon \gamma_*Y\to G/\tau(H)\quad \text{such that}\quad  g \uT \gamma_!(y^{(0)})\ \longmapsto\ g\cdot \tau(H).
\end{equation}
We set
\begin{equation}\label{e:nu-lam-gam*}
\nu=\lambda\circ \gamma_!\colon\ G/H=Y\ \labelto{\gamma_!}\ \gamma_*Y\ \labelto{\lambda}\  G/\tau(H)\hs,
\end{equation}
where $\gamma_!\colon Y\to\gamma_*Y$ is the $\gamma$-semi-morphism of Example \ref{ex:gamma-ast}.
Then we have $\nu_\natural=\lambda$.
Since $\nu_\natural$ is an isomorphism of $G$-$k$-varieties,
by Lemma \ref{l:equi}  \,$\nu$ is a $\tau$-equivariant $\gamma$-semi-isomorphism.
Since by \eqref{e:lambda-gamma-star} and \eqref{e:nu-lam-gam*} we have
\[\nu(1\cdot H)=\nu(y^{(0)})=\lambda(\gamma_!(y^{(0)}))=\lambda(1\uT\gamma_!(y^{(0)}))=1\cdot\tau(H)\]
and since $\nu$ is $\tau$-equivariant, we obtain that
\[\nu(g\cdot H)=\nu(g\cdot (1\cdot H))=\tau(g)\cdot \nu(1\cdot H)=\tau(g)\cdot\tau(H),\]
which proves (ii).
 \end{proof}

\begin{corollary}\label{c:maincor-3}
Let $G$ be a linear algebraic $k$-group and $H\subset G$ be an algebraic $k$-subgroup. Set $Y=G/H$.
Let $\gamma\in\Aut(k)$ and let $\tau\colon G\to G$ be a $\gamma$-semi-automorphism of $G$.
The following three conditions are equivalent:
\begin{enumerate}
\item[\rm(i)] There exists a $\tau$-equivariant $\gamma$-semi-automorphism $\mu\colon Y\to Y$;
\item[\rm(ii)] The $G$-$k$-variety $G/\tau(H)$ is isomorphic to $G/H$;
\item[\rm(iii)] The algebraic subgroup $\tau(H)\subset G$ is conjugate to $H$.
\end{enumerate}
\end{corollary}

\begin{proof}
By Corollary \ref{c:tau-equi-semi} there exists $\mu\colon Y\to Y$ as in (i) if and only if
the $G$-$k$-variety $(\gamma_*Y,\tau^*\gamma_*\theta)$ is isomorphic to $(Y,\theta)$.
By construction $(Y,\theta)=G/H$, and by Lemma \ref{l:nu-1-2}, $(\gamma_*Y,\tau^*\gamma_*\theta)\cong G/\tau(H)$.
Thus (i)$\Leftrightarrow$(ii).
By Lemma \ref{l:isom-conj} (ii)$\Leftrightarrow$(iii).
 \end{proof}

\begin{remark}
Lemma \ref{l:nu-1-2} and Corollary \ref{c:maincor-3}
are the main results of Sections  \ref{s:conj}--\ref{s:homog}.
\end{remark}

\section{Equivariant models of $G$-varieties}
\label{s:models}

\begin{subsec}\label{ss:k-k0-equi}
Let $k$ be an algebraically closed field,
and $k_0\subset k$ be a subfield such that $k$ is  a Galois extension of $k_0$,
that is, $k_0$ is a perfect field and $k$ is an  algebraic closure of $k_0$.
We write $\Gamma=\Gal(k/k_0):=\Aut(k/k_0)$.

Let $Y$ be a $k$-variety.
Let $\gamma,\gamma'\in\Gamma$.
If $\mu$ is a $\gamma$-semi-automorphism of $Y$, and $\mu'$ is a $\gamma'$-semi-automorphism of $Y$,
then $\mu\circ\mu'$ is a $\gamma\circ\gamma'$-semi-automorphism of $Y$
and $\mu^{-1}$ is a $\gamma^{-1}$-semi-automorphism of $Y$.

We denote by $\SAut_{k/k_0}(Y)$ or just by $\SAut(Y)$ the group
of all $\gamma$-semi-auto\-morphisms  $\mu$ of $Y$ where $\gamma$
runs over $\Gamma=\Gal(k/k_0)$.

A {\em $k_0$-model of $Y$} is a $k_0$-variety $Y_0$ together with an isomorphism of $k$-varieties
\[\vk_Y\colon Y_0\times_{k_0} k\isoto Y.\]
Note that $\gamma\in \Gamma$ defines a $\gamma$-semi-automorphism of $Y_0\times_{k_0} k$
$$\id_{Y_0}\times(\gamma^*)^{-1} \colon \ Y_0\times_{\Spec k_0} \Spec k\  \longrightarrow\ \ Y_0\times_{\Spec k_0} \Spec k$$
and thus, via $\vk_Y$, a $\gamma$-semi-automorphism $\mu_\gamma$ of $Y$.
We obtain a homomorphism
\[\Gamma\to \SAut(Y),\ \gamma\mapsto\mu_\gamma\hs .\]
\end{subsec}
Conversely:

\begin{lemma}[\rm Borel and Serre {\cite[Lemma 2.12]{Borel-Serre}}]
\label{l:BS}
Let $k$, $k_0$, $\Gamma$, $Y$ be as above.
Assume that for every $\gamma\in\Gamma$ we have a $\gamma$-semi-automorphism $\mu_\gamma$
of $Y$ such that the following conditions are satified:
\begin{enumerate}
\item[\rm(i)]  the map $\Gamma\to\SAut_{k/k_0}(Y),\ \gamma\mapsto\mu_\gamma$, is a homomorphism,
\item[\rm(ii)]  the restriction of this map
to $\Gal(k/k_1)$ for some finite Galois extension $k_1/k_0$ in $k$
comes from a $k_1$-model $Y_1$ of $Y$,
\item[\rm(iii)] $Y$ is quasi-projective.
\end{enumerate}
Then there exists a $k_0$-model $Y_0$ of $Y$ that defines this homomorphism $\gamma\mapsto\mu_\gamma$.\qed
\end{lemma}

\begin{subsec}\label{ss:5.2}
Let $k$ and $k_0$ be as in \ref{ss:k-k0-equi}, and
let $G$ be a linear algebraic group over $k$.
We denote by $\SAut_{k/k_0}(G)$ or just by $\SAut(G)$ the group
of all $\gamma$-semi-auto\-morphisms  $\sigma$ of $G$ where $\gamma$
runs over $\Gamma=\Gal(k/k_0)$.

We assume that we are given a $k_0$-model of $G$,
 that is, a linear algebraic group $G_0$ over $k_0$
together with an isomorphism of algebraic $k$-groups
$\vk_G\colon G_0\times_{k_0} k\isoto G$.
For $\gamma\in\Gamma$, the automorphism $(\gamma^*)^{-1}$ of $\Spec k$ induces a $\gamma$-semi-automorphism
$\id_{G_0}\times(\gamma^*)^{-1}$ of $G_0\times_{\Spec k_0} \Spec k$.
We identify $G$ with $G_0\times_{\Spec k_0} \Spec k$ via $\vk_G$;
then for any $\gamma\in \Gamma$ we obtain a $\gamma$-semi-automorphism $\sigma_\gamma\colon G\to G$.
The map
\[\Gamma\to\SAut(G),\quad \gamma\mapsto\sigma_\gamma\]
is a homomorphism.
We identify $\gamma_* G$ with $G$ using $(\sigma_\gamma)_\natural\colon \gamma_*G\to G$.

For a $G$-$k$-variety $(Y,\theta)$ and for a given $k_0$-model $G_0$  of $G$,
we ask whether there exists a  $G_0$-equivariant  $k_0$-model $(Y_0,\theta_0)$ of $(Y,\theta)$.

Let  $(Y,\theta)$ be a $G$-$k$-variety.
We write $g\cdot y$ for $\theta(g,y)$.
Recall (Definition \ref{d:g-semi-morphism}) that a $\gamma$-semi-automorphism $\mu$
of $Y$ is {\em $\sigma_\gamma$-equivariant} if
the following diagram commutes:
\begin{equation*}
\xymatrix{
G\times_k Y\ar[r]^-\theta\ar[d]_{\sigma_\gamma\times\mu}   & Y \ar[d]^\mu\\
G\times_k Y\ar[r]^-\theta                                     &Y
}
\end{equation*}
This is the same as to require that
\[\mu(g\cdot y)=\sigma_\gamma(g)\cdot\mu(y)\quad\text{for all }g\in G(k),\ y\in Y(k).\]

A  $G_0$-equivariant  $k_0$-model $(Y_0,\theta_0)$ of $(Y,\theta)$ defines a homomorphism
\[\Gamma\to \SAut_{k/k_0}(Y),\quad\gamma\mapsto\mu_\gamma\hs ,\]
where for every $\gamma\in\Gamma$, the $\gamma$-semi-automorphism
$\mu_\gamma$ of $Y$ is $\sigma_\gamma$-equivariant.
Conversely:
\end{subsec}

\begin{lemma}\label{l:model}
Let $k$, $k_0$, $\Gamma=\Gal(k/k_0)$, $G$, $(Y,\theta)$ be as in \ref{ss:5.2} and
let $G_0$ be a $k_0$-model of $G$.
Assume that for every $\gamma\in\Gamma$ we have a $\gamma$-semi-automorphism $\mu_\gamma$
of $Y$ such that the following conditions are satisfied:
\begin{enumerate}
\item[\rm(i)]  the map $\Gamma\to\SAut_{k/k_0}(Y),\ \gamma\mapsto\mu_\gamma$ is a homomorphism,
\item[\rm(ii)]  the restriction of this map
to $\Gal(k/k_1)$ for some finite Galois extension $k_1/k_0$ in $k$
comes from a $G_1$-equivariant $k_1$-model $Y_1$ of $Y$, where $G_1=G_0\times_{k_0} k_1$,
\item[\rm(iii)] $Y$ is quasi-projective,
\item[\rm(iv)] for every $\gamma\in\Gamma$, the $\gamma$-semi-automorphism
$\mu_\gamma$ is $\sigma_\gamma$-equivariant.
\end{enumerate}
Then there exists a  $G_0$-equivariant $k_0$-model $(Y_0,\theta_0)$ of $(Y,\theta)$
that defines this homomorphism $\gamma\mapsto\mu_\gamma$.
\end{lemma}

\begin{proof}
By Lemma \ref{l:BS}  the homomorphism
\[\Gamma\to\SAut(Y),\quad \gamma\mapsto \mu_\gamma\]
defines a $k_0$-model $Y_0$ of  $Y$.
Using Galois descent for morphisms (see, for example, Jahnel \cite[Proposition 2.8]{Jahnel})
we obtain from condition (iv)
that $\theta$ comes from some morphism $\theta_0\colon G_0\times_{k_0}Y_0\to Y_0$,
and the $k_0$-model $(Y_0,\theta_0)$ of $(Y,\theta)$ is $G_0$-equivariant.
 \end{proof}

\begin{remark}
If in Lemma \ref{l:model} we do not assume that $Y$ is quasi-projective,
then we obtain a $k_0$-model $Y_0$
in the category of algebraic $k_0$-spaces (see Wedhorn \cite[Proposition 8.1]{Wedhorn}\hs),
but not necessarily in the category of $k_0$-schemes (even when $k=\C$ and $k_0=\R$,
see  Huruguen \cite[Theorem 2.35]{Huruguen}).
\end{remark}

We need a proposition.

\begin{proposition}[\rm well-known; see, for example, EGA IV$_3$ {\cite[Theorem 8.8.2]{EGA}}]
\label{p:anon}
Let $k$ be an algebraically closed field.
\begin{enumerate}
\item[\rm(i)] For any $k$-variety $X$ and any subfield $k_0$ of $k$,
there exists a $k_1$-model $X_1$ of $X$ for some  finitely generated extension $k_1$ of $k_0$ in $k$;
\item[\rm(ii)] If $k_1$ is a subfield of algebraically closed field
$k$, $f\colon X\to Y$ a morphism of $k$-varieties,
and $X_1$, $Y_1$ are $k_1$-models of $X$, $Y$, respectively,
then there exists a finitely generated field extension $k_2$ of $k_1$ in $k$ such that,
if we set  $X_2=X_1\times_{k_1} k_2$ and  $Y_2=Y_1\times_{k_1} k_2$,
then there exists a $k_2$-morphism $f_2\colon X_2\to Y_2$
such that the triple $(X_2,Y_2,f_2)\times_{k_2} k$ is isomorphic to $(X,Y,f)$.
\end{enumerate}
\end{proposition}

\begin{proof}[Proof {\rm (communicated by an anonymous  MathOverflow user)}]

(i) A variety $X$ is a finite union of affine open subvarieties $X_i$.
Since $X$ is separated, the intersections $X_i\cap X_j$ are  affine varieties;
see, for example, Liu \cite[Ch.~3, Proposition 3.6 on p.~100]{Liu}.
Now $X$ can be reconstructed from the affine varieties $X_i,X_i\cap X_j$
and the morphisms of affine varieties $X_i\cap X_j\to X_i$.
Obviously, this system is defined over a subfield of $k$  finitely generated over $k_0$.

(ii) The graph of $f$ is a closed subvariety of $X\times Y$,
and so is defined by an ideal in the structure sheaf of $X\times Y$,
which is obviously defined over a subfield of $k$
finitely generated over $k_1$.
 \end{proof}

\begin{corollary}\label{c:anon-anon}
Let $G$ be a linear algebraic group over an algebraically closed field $k$,
and let $\theta\colon G\times _k Y\to Y$ be an action of $G$ on a $k$-variety $Y$.
Let $k_0$ be a subfield of $k$, and let $G_0$ be a $k_0$-model of $G$.
Then there exists a finitely generated extension $k_2$ of $k_0$ in $k$,
a $k_2$-variety $Y_2$, and a $k_2$-action $\theta_2\colon G_2\times_{k_2} Y_2\to Y_2$
such that $(Y_2,\theta_2)$ is a $k_2$-model of $(Y,\theta)$. Here $G_2=G_0\times_{k_0} k_2$.
\end{corollary}

\begin{proof} By Proposition \ref{p:anon}(i) there exists a $k_1$-model  $Y_1$ of the variety
$Y$ for some finitely generated extension $k_1$ of $k_0$ in $k$.
We obtain a $k_1$-model $G_1\times_{k_1} Y_1$ of $G\times_k Y$, where $G_1=G_0\times_{k_0} k_1$.
By Proposition \ref{p:anon}(ii) the action $\theta$
can be defined over a finitely generated extension $k_2$ of $k_1$ in $k$.
 \end{proof}

\begin{lemma}\label{l:Aut-1}
Let $k$, $k_0$, $\Gamma=\Gal(k/k_0)$, $G$, $(Y,\theta)$ be as in \ref{ss:5.2}, and
let $G_0$ be a $k_0$-model of $G$.
{\em Assume that}  $\Aut^G(Y)=\{1\}$.
Assume that for every $\gamma\in\Gamma$ there exists a $\gamma$-semi-automorphism $\mu_\gamma$ of $Y$
satisfying condition (iv) of Lemma \ref{l:model}.
Then such $\mu_\gamma$ is unique, and the map $\gamma\mapsto\mu_\gamma$
satisfies conditions (i) and (ii) of Lemma \ref{l:model}.
\end{lemma}

\begin{proof}
If $\mu'_\gamma$ another such $\gamma$-semi-automorphism,
then $\mu_\gamma^{-1} \mu'_\gamma\in \Aut^G(Y)=\{1\}$,
hence $\mu'_\gamma=\mu_\gamma$.
If $\beta,\gamma\in\Gamma$, then $\mu_{\beta\gamma}^{-1}\mu_\beta\hs \mu_\gamma\in \Aut^G(Y)=\{1\}$,
hence $\mu_{\beta\gamma}=\mu_\beta\hs \mu_\gamma$, hence the map $\gamma\mapsto\mu_\gamma$ is a homomorphism,
that is, condition (i) is satisfied.

By Corollary \ref{c:anon-anon}, there exists a finite field extension $k_1/k_0$ in $k$
and a  $G_1$-equivariant $k_1$-model $(Y_1,\theta_1)$ of $(Y,\theta)$, where  $G_1=G_0\times_{k_0} k_1$.
This $k_1$-model defines a homomorphism
\begin{equation*}
\gamma\mapsto \mu'_\gamma\colon\ \Gal(k/k_1)\to \SAut(Y)
\end{equation*}
such that $\mu'_\gamma$ is $\sigma_\gamma$-equivariant for all $\gamma\in\Gal(k/k_1)$.
Since a $\sigma_\gamma$-equivariant $\gamma$-semi-automorphism is unique,
we see that for all $\gamma\in\Gal(k/k_1)$
we have $\mu_\gamma=\mu'_\gamma$, and hence, the restriction of the map
\begin{equation*}
\gamma\mapsto \mu_\gamma\colon \Gamma\to \SAut(Y)
\end{equation*}
to $\Gal(k/k_1)$ comes from the $k_1$-model  $(Y_1,\theta_1)$ of $(Y,\theta)$,
that is, condition (ii) of Lemma \ref{l:model} is satisfied.
 \end{proof}

\section{Spherical homogeneous spaces and their combinatorial invariants}
\label{s:invariants}

Starting this section, $k$ is an algebraically closed field {\em of characteristic} 0.
Let $G$ be a connected  reductive $k$-group.
We describe combinatorial invariants (invariants of Luna and Losev)
of a spherical homogeneous space $Y=G/H$ of $G$.

\begin{subsec}
We start with combinatorial invariants of $G$.
We fix $T\subset B\subset G$, where $B$ is a Borel subgroup and $T$ is a maximal torus.
Let $\BRD(G)=\BRD(G,T,B)$ denote the based root datum of $G$.
We have
\[ \BRD(G,T,B)=(X,X^\vee, R,R^\vee, S, S^\vee)\]
where\\
 $X=\X^*(T):=\Hom(T,\G_{m,k})$ is the character group of $T$;\\
 $X^\vee=\X_*(T):=\Hom(\G_{m,k}, T)$ is the cocharacter group of $T$;\\
 $R=R(G,T)\subset X$ is the root system;\\
 $R^\vee\subset X^\vee$ is the coroot system;\\
 $S=S(G,T,B)\subset R$ is the system of simple roots (the basis of $R$) defined by $B$;\\
 $S^\vee\subset R^\vee$ is the system of simple coroots.

 There is a canonical pairing $X\times X^\vee\to \Z,\ (\chi,x)\mapsto \langle \chi,x\rangle$,
 and a canonical bijection $\alpha\mapsto \alpha^\vee\colon R\to R^\vee$
 such that $S^\vee=\{\alpha^\vee\ |\ \alpha\in S\}$.
See Springer \cite[Sections 1 and 2]{Springer} for details.

We consider also the Dynkin diagram $\Dyn(G)=\Dyn(G,T,B)$,
which is a graph with the set of vertices $S$.
The edge between two simple roots $\alpha,\beta\in S$
is described in terms of the integers $\langle \alpha,\beta^\vee\rangle$
and $\langle \beta,\alpha^\vee\rangle$.

We call a pair $(T,B)$ as above a {\em Borel pair.}
If $(T',B')$ is another Borel pair, then
by  Theorem 11.1 and Theorem 10.6(4) in Borel's book \cite{Borel},
there exists $g\in G(k)$ such that
\begin{equation}\label{e:Borel-pair}
g\cdot T\cdot g^{-1}=T',\quad  g\cdot B\cdot g^{-1}=B'.
\end{equation}
This element $g$ induces an isomorphism
\[ g^*\colon \BRD(G,T',B')\isoto\BRD(G,T,B).\]
If $g'\in G(k)$ another element as in \eqref{e:Borel-pair}, then $g=gt$ for some $t\in T(k)$,
and therefore, the isomorphism
\[ (g')^*\colon \BRD(G,T',B')\isoto\BRD(G,T,B)\]
coincides with $g^*$.
Thus we may  identify the based root datum $\BRD(G,T',B')$ with $\BRD(G,T,B)$
and write $\BRD(G)$ for $\BRD(G,T,B)$.
We say that $\BRD(G)$ is the {\em canonical based root datum of} $G$.
We see that the  based root datum $\BRD(G)$ is an invariant of $G$.
In particular, the character lattice $X=\X^*(T)$ with the subset $S\subset X$ is an invariant,
and the Dynkin diagram $\Dyn(G)$ is an invariant.
\end{subsec}

\begin{subsec}
We describe the combinatorial invariants of a homogeneous spherical $G$-variety $Y=G/H$.
Let $K(Y)$ denote the field of rational functions of $Y$.
The group $G(k)$ acts on $K(Y)$ by
\[ (g\cdot f)(y) = f(g^{-1}\cdot y)\quad \text{for }f\in K(Y),\ g\in G(k), \text{ and } y\in Y(k).\]
For $\chi\in\X^*(B)$, let
 $K(Y)^{(B)}_\chi$ denote the space of $\chi$-eigenfunctions in $K(Y)$, that is,
 the $k$-space of rational functions $f\in K(X)$ such that
\[ b\cdot f=\chi(b)\cdot f\quad \text{for all } b\in B(k).\]
Since $B$ has an open dense orbit in $Y$, the $k$-dimension of $K(Y)^{(B)}_\chi$ is  at most 1.
Let $\sX=\sX(Y)\subset\X^*(B)$ denote the set of characters $\chi$ of $B$ such that $K(Y)^{(B)}_\chi\neq \{0\}$.
Then $\sX$ is a subgroup of $\X^*(B)$ called the {\em weight lattice of $Y$.}
We set
\[ V=V(Y)=\Hom_\Z(\sX,\Q).\]

Let $\Val(K(Y))$ denote the {\em set of $\Q$-valued valuations}
of the field $K(Y)$ that are trivial on $k$.
The group  $G(k)$ naturally acts on $K(Y)$ and on $\Val(K(Y))$.
We shall consider
the set $\Val^B(K(Y))$ of $B(k)$-invariant valuations and
the set $\Val^G(K(Y))$ of $G(k)$-invariant valuations.
We have a canonical map
\[\rho\colon \Val(K(Y))\to V,\ v\mapsto (\chi\mapsto v(f_\chi)\hs ),\]
where $v\in\Val(K(Y)),\ \chi\in\sX,\ f_\chi\in K(Y)^{(B)}_\chi$, $f_\chi\neq 0$.
It is known  (see Knop \cite[Corollary 1.8]{Knop-LV}\hs)
that the restriction of $\rho$ to $\Val^G(K(Y))$ is injective.
We denote by
\[\sV=\sV(Y):=\rho(\Val^G(K(Y)))\subset V\]
the image of $\Val^G(K(Y))$ in $V$.
It is a cone in $V$ called the {\em valuation cone of $Y$}; see Perrin \cite[Corollary 3.3.6]{Perrin}.

Let $\sD=\sD(Y)$ denote the {\em set of colors of $Y$}, that is,
the set of $B$-invariant prime divisors in $Y$.
Each color $D\in\sD$ defines a $B$-invariant valuation of $K(Y)$ that we denote by $\val(D)$.
Thus we obtain a map
\[\val\colon \sD\to\Val(K(Y)).\]
By abuse of notation we denote $\rho(\val(D))\in V$ by $\rho(D)$.
Thus we obtain a map
\[\rho\colon \sD\to V,\]
which in general is not injective (for example, it is not injective
for $G$ and $Y$ as in Example \ref{ex:CF}).

For $D\in\sD$, let $\Stab_G(D)$ denote the stabilizer of $D\subset Y$ in $G$.
Clearly $\Stab_G(D)\supset B$, hence $\Stab_G(D)$ is a parabolic subgroup of $G$.
For $\alpha\in S$, let $P_\alpha\supset B$
denote the corresponding minimal parabolic subgroup of $G$ containing $B$.
Let $\vs(D)$ denote the set of $\alpha\in S$ for which $P_\alpha$ is {\em not contained} in $\Stab_G(D)$.
We obtain a map
\[\vs\colon \sD\to\sP(S),\]
where $\sP(S)$ denotes the set of all subsets of $S$.
\end{subsec}

\begin{lemma}\label{l:fiber-2}
Any fiber of the map $\vs$ has $\le 2$ elements.
\end{lemma}

\begin{proof}
Let $D\in\sD$.
Since  $Y$ is a homogeneous $G$-variety, clearly $\Stab_G(D)\neq G$, hence $\vs(D)\neq\varnothing$.
We see that there exists $\alpha\in \vs(D)$.
Consider the set $\sD(\alpha)$ consisting of those $D\in \sD$ for which $\alpha\in\vs(D)$.
By Proposition \ref{p:facts-type-a} in Appendix \ref{s:App} below
we have $|\sD(\alpha)|\le 2$, and the lemma follows.
 \end{proof}

Consider the map
\[\rho\times\vs\ \colon \sD\to V\times\sP(S).\]

\begin{corollary}\label{c:fiber-2}
Any fiber of the map $\rho\times\vs$ has $\le 2$ elements.\qed
\end{corollary}

Consider the subset $\Omega:=\im(\rho\times \vs)\subset V\times\sP(S)$.
Let $\Omega^{(2)}$ (resp.~$\Omega^{(1)}$) denote the subset of $\Omega$
consisting of the elements with two preimages (resp.~with one preimage) in $\sD$.
We obtain two subsets $\Omone,\Omt\subset V\times\sP(S)$,
and by Corollary \ref{c:fiber-2} we have  $\Omega=\Omone\sqcup\hs \Omt$ (disjoint union).

\begin{definition}\label{d:comb-inv}
Let $G$ be a connected reductive group over an algebraically closed field $k$ of characteristic 0.
Let $Y=G/H$ be a spherical homogeneous space of $G$.
By the {\em combinatorial invariants of $Y$} we mean
\[ \sX\subset\X^*(B),\quad \sV\subset V:=\Hom_\Z(\sX,\Q), \text{ and } \Omone,\Omt\subset V\times\sP(S).\]
\end{definition}

\begin{subsec}\label{ss:for-Losev}
Let  $G$ be a connected reductive $k$-group.
Let  $H_1\subset G$ be a spherical subgroup; then we set $Y_1=G/H_1$.
We consider the set of colors $\sD(Y_1)$ and the canonical maps
\[ \rho_1\colon \sD(Y_1)\to  V(Y_1),\quad \varsigma_1\colon \sD(Y_1)\to \sP(S).\]
If $H_2\subset G$ is another spherical subgroup, then we set $Y_2=G/H_2$
and consider the set of colors $\sD(Y_2)$ and the canonical maps
\[ \rho_2\colon \sD(Y_2)\to  V(Y_2),\quad \varsigma_2\colon \sD(Y_2)\to \sP(S).\]
Now assume that there exists $a\in G(k)$ such that $H_2=a H_1 a^{-1}$.
Then we have an isomorphism of $G$-varieties of Lemma \ref{l:isom-conj}
\[\phi_a\colon Y_1\to Y_2,\quad g\cdot H_1\mapsto ga^{-1}\cdot H_2\,.\]
It follows that $\sX(Y_1)=\sX(Y_2)$, $V(Y_1)=V(Y_2)$, $\sV(Y_1)=\sV(Y_2)$.
Moreover, the $G$-equivariant map $\phi_a\colon Y_1\to Y_2$
induces a bijection
\[(\phi_a)_*\colon\, \sD(Y_1)\to \sD(Y_2)\]
satisfying
\begin{equation}\label{e:a-rho-vs}
\rho_2\circ(\phi_a)_*=\rho_1,\quad \vs_2\circ(\phi_a)_*=\vs_1\,.
\end{equation}
It follows that
\[\Omone(Y_1)=\Omone(Y_2)\quad\text{and}\quad \Omt(Y_1)=\Omt(Y_2).\]
Conversely:
\end{subsec}

\begin{proposition}
[\rm Losev's\hm\ Uniqueness\hm\ Theorem\hm\ {\cite[Theorem 1]{Losev}}]
\label{t:Losev}
Let $G$ be a connected reductive group over an algebraically closed field $k$ {of characteristic $0$}.
Let $H_1,H_2\subset G$ be two spherical subgroups, and let $Y_1=G/H_1$ and $Y_2=G/H_2$
be the corresponding spherical homogeneous spaces.
If $\sX(Y_1)=\sX(Y_2)$, $\sV(Y_1)=\sV(Y_2)$, $\Omone(Y_1)=\Omone(Y_2)$, and $\Omt(Y_1)=\Omt(Y_2)$,
then there exists $a\in G(k)$ such that $H_2=a H_1 a^{-1}$.
\end{proposition}

\begin{subsec}
Consider the group $\Aut^G(Y)=\mathcal N_G(H)/H$ (cf.~Corollary \ref{l:k-automorphisms}).
This group acts on $\mathcal D$.
We consider the surjective  map
\begin{equation}\label{e:zeta0}
\zeta=\rho\times\vs\colon \sD\to \Omega.
\end{equation}
By Corollary \ref{c:fiber-2} each of the fibers of $\zeta$  has either one or two elements.
We denote by $\Aut_\Omega(\sD)$ the group of permutations
$\pi\colon \sD\to\sD$ such that $\zeta\circ \pi=\zeta$.
From \eqref{e:a-rho-vs} we see that the group $\Aut^G(Y)$, when acting on $\sD$,
acts on the fibers of the map $\zeta$, and so we obtain a homomorphism
\begin{equation*}
\Aut^G(Y)\to \Aut_\Omega(\sD).
\end{equation*}
\end{subsec}

\begin{theorem}[\rm Losev, unpublished]
\label{t:Losev-3}
Let $G$ be a connected reductive group over an algebraically closed field $k$ of characteristic 0.
Let $Y=G/H$ be a spherical homogeneous space of $G$.
Then, with the above notation,  the homomorphism
\begin{equation}\label{e:zeta}
\Aut^G(Y)\to \Aut_\Omega(\sD).
\end{equation}
 is surjective.
\end{theorem}

This theorem will be proved in Appendix \ref{s:App}; see Theorem \ref{t:Losev-3-bis}.

\begin{corollary}[\rm Strong version of Losev's Uniqueness Theorem]
Let $G$,\,$H_1$,\,$H_2$, $Y_1=G/H_1$, $Y_2=G/H_2$ be as in Proposition \ref{t:Losev},
in particular, $\sX(Y_1)=\sX(Y_2)$, $\sV(Y_1)=\sV(Y_2)$,
$\Omone(Y_1)=\Omone(Y_2)$, and $\Omt(Y_1)=\Omt(Y_2)$.
Let $\varphi\colon \sD(Y_1)\to\sD(Y_2)$ be {\emm any} bijection satisfying
\[\rho_2\circ\varphi=\rho_1,\quad \vs_2\circ\varphi=\vs_1\]
(such a bijection exists because $\Omone(Y_1)=\Omone(Y_2)\text{ and } \Omt(Y_1)=\Omt(Y_2)$\hs).
Then there exists $a'\in G(k)$ such that $H_2=a' H_1(a')^{-1}$ and
\[(\phi_{a'})_*=\varphi\colon\, \sD(Y_1)\to \sD(Y_2).\]
\end{corollary}

\begin{proof}
By Proposition  \ref{t:Losev} there exists $a\in G(k)$ such that $H_2=a H_1 a^{-1}$.
This element $a$ defines a map
\[(\phi_a)_*\colon \sD(Y_1)\to\sD(Y_2)\]
satisfying \eqref{e:a-rho-vs}.
Set
\[\psi=(\phi_a)_*^{-1}\circ\varphi\colon \  \sD(Y_1)\to\sD(Y_1);\]
then $\psi$ satisfies
\[\rho_1\circ\psi=\rho_1,\quad \vs_1\circ\psi=\vs_1\,,\]
hence $\psi\in\Aut_\Omega\, \sD(Y_1).$
By Theorem \ref{t:Losev-3} there exists $n\in\sN_G(H_1)$
such that
\[(\phi_n)_*=\psi\colon \sD(Y_1)\to \sD(Y_1).\]
We set $a'=an$; then $a' H_1(a')^{-1}=H_2$, \
$\phi_{a'}=\phi_a\circ\phi_n$, \ and
\[(\phi_{a'})_*=(\phi_a)_*\circ (\phi_n)_*=(\phi_a)_*\circ\psi=\varphi.\qedhere\]
\end{proof}

\begin{corollary}\label{c:closed-isom}
In Theorem \ref{t:Losev-3}, if  $H$  is spherically closed,
then the homomorphism \eqref{e:zeta} is an isomorphism.
\end{corollary}

\begin{proof}
Indeed, by the definition of a spherically closed spherical subgroup,
the homomorphism \eqref{e:zeta} is injective,
and by Theorem \ref{t:Losev-3} it is surjective, hence it is bijective, as required.
 \end{proof}

Note that
\[  \Aut_\Omega(\sD)=\prod_{\omega\in \Omega} \Aut(\zeta^{-1}(\omega)),\]
where $\Aut(\zeta^{-1}(\omega))$ is the group of permutations of the set $\zeta^{-1}(\omega)$.
It is clear that  for every $\omega\in \Omega$, the restriction homomorphism
\begin{equation}\label{e:Aut-Omega-sD}
\Aut_\Omega(\sD)\to \Aut(\zeta^{-1}(\omega))
\end{equation}
is surjective.

\begin{corollary}\label{c:D-injective}
In Theorem \ref{t:Losev-3}, if  $\NGH=H$, then the surjective map
$\zeta$ of \eqref{e:zeta0} is bijective, hence $\sD$ injects into $V\times\sP(S)$.
\end{corollary}

\begin{proof}
It follows from Theorem \ref{t:Losev-3} and the surjectivity of the homomorphism \eqref{e:Aut-Omega-sD},
that the group $\Aut^G(Y)=\NGH/H$ acts transitively
on the fiber $\zeta^{-1}(\omega)$ for every $\omega\in \Omega$.
Since by assumption $\NGH/H=\{1\}$, we conclude that each fiber of $\zeta$ has exactly one element,
hence $\zeta$ is bijective, as required.
 \end{proof}

\section{Action of an automorphism of the base field\\
on the combinatorial invariants of a spherical homogeneous space}
\label{s:action}

\begin{subsec}\label{ss:H1-H2}
Let $k$ be an algebraically closed field of characteristic 0,
$G$ be a  connected  reductive group over $k$,
$H_1\subset G$ be a spherical subgroup, and
$Y_1=G/H_1$ be the corresponding spherical homogeneous space.

Let $k_0\subset k$ be a subfield and  let $\gamma\in\Aut(k/k_0)$.
We assume that ``$G$ is defined over $k_0$", that is, we are given a $k_0$-model $G_0$ of $G$.
Then we have a $\gamma$-semi-automorphism $\sigma_\gamma$ of $G$; see Subsection \ref{ss:5.2}.
Set $H_2=\sigma_\gamma(H_1)\subset G$ and denote by $Y_2:=G/H_2$
the corresponding spherical homogeneous space.

We wish to know whether the spherical homogeneous spaces $Y_1$ and $Y_2$ are isomorphic as $G$-varieties.
For this end we compare their combinatorial invariants.
\end{subsec}

We fix a Borel pair $(T,B)$; then $T\subset B\subset G$.
Consider
\[ \sigma_\gamma(T)\subset\sigma_\gamma(B) \subset G.\]
Then $(\sigma_\gamma(T),\sigma_\gamma(B))$ is again a Borel pair and  hence,
there exists $g_\gamma\in G(k)$ such that
\[ g_\gamma\cdot\sigma_\gamma(T)\cdot g_\gamma^{-1}=T,\quad  g_\gamma\cdot\sigma_\gamma(B)\cdot g_\gamma^{-1}=B.\]
Set $\tau={\rm inn}(g_\gamma)\circ\sigma_\gamma\colon\  G\to G$; then $\tau$
is a $\gamma$-semi-automorphism of $G$, and
\begin{equation}\label{e:sig-T-B}
 \tau(B)=B,\quad \tau(T)=T.
\end{equation}
Set $H'_2=\tau(H_1)\subset G$ and $Y'_2=G/H'_2$.
We have $H'_2=g_\gamma\cdot H_2\cdot g_\gamma^{-1}$, so by Lemma \ref{l:isom-conj} $Y_2$ and $Y_2'$ are isomorphic, and
we wish to know whether $Y_1$ and $Y'_2$ are isomorphic.

By \eqref{e:sig-T-B},  $\tau$ naturally acts on the characters of $T$ and $B$;
we denote the corresponding automorphism by $\veg$.
By definition
\begin{equation}\label{e:veg}
\veg(\chi)(b)=\gamma(\chi(\hs \tau^{-1}(b)\hs )\hs ) \quad\text{for }\chi\in\X^*(B),\ b\in B(k),
\end{equation}
and the same holds for the characters of $T$ (recall that $\X^*(B)=\X^*(T)$\hs ).
Since $\tau(B)=B$, we see that $\veg$, when acting on $\X^*(T)$,
preserves the set of simple roots $S=S(G,T,B)\subset \X^*(T)$.
It is well known (see for example \cite[Section 3.2 and Proposition 3.1(a)]{BKLR})
that the automorphism $\veg$ does not depend on the choice of $g_\gamma$ and that the map
\[\ve\colon \Aut(k/k_0)\to \Aut(\X^*(T), S),\quad \gamma\mapsto \veg \]
is a homomorphism.
Since $\veg$ acts on $\X^*(B)$ and on $S$, one can define
$\veg (\sX(Y_1))$, $\veg(\sV(Y_1))$, $\veg(\Omone(Y_1)\hs )$, $\veg(\Omt(Y_1)\hs )$.

Following Akhiezer \cite{Akhiezer}, we compute the combinatorial invariants of the spherical homogeneous space $Y'_2$.
We define a map on $k$-points
\begin{equation}\label{e:nu-1-2'}
 Y_1(k)\to Y_2'(k),\quad g\cdot H_1\mapsto \tau(g)\cdot H_2',\text{ where } g\in G(k).
\end{equation}
Recall that $H_2'=\tau(H_1)$.
It is easy to see that  the map $\nu$ is well-defined and bijective and that we have
\begin{equation}\label{e:nu-tau}
\nu^{-1}(g'\cdot y_2')=\tau^{-1}(g')\cdot\nu^{-1}(y_2')\quad\text{for }g'\in G(k),\ y_2'\in Y_2'(k).
\end{equation}
By  Lemma \ref{l:nu-1-2} the map \eqref{e:nu-1-2'} is induced
by some $\tau$-equivariant  $\gamma$-semi-isomorphism
\[ \nu\colon Y_1\to Y_2'\hs ,\]
and so we obtain an isomorphism of the fields of rational functions
\[\nu_*\colon K(Y_1)\to K(Y_2').\]

\begin{lemma}\label{l:Akhiezer-calc}
Let $\chi_1\in\X^*(B)$ and assume that $f_1\in K(Y_1)^{(B)}_{\chi_1}$\hs.
Then $\nu_*f_1\in K(Y_2')^{(B)}_{\chi_2}$, where $\chi_2=\veg(\chi_1)$.
\end{lemma}

\begin{proof}
Since $f_1\in K(Y_1)^{(B)}_{\chi_1}$\hs, we have
\begin{equation}\label{e:b-m-y}
f_1(b^{-1} y_1)=\chi_1(b)\cdot f_1(y_1)\quad\text{for all }\ y_1\in Y_1(k),\ b\in B(k).
\end{equation}
We write $f_2'=\nu_* f_1\in K(Y_2')$.
Since $\nu\colon Y_1\to Y_2'$ is a $\gamma$-semi-isomorphism, by Corollary \ref{l:nu-*}
we have
\begin{equation}\label{e:f'2}
f_2'(y_2')=\gamma(f_1(\nu^{-1}(y_2')))\quad\text{for } y_2'\in Y_2'(k).
\end{equation}
Note that $\tau^{-1}\colon G\to G$ is a $\gamma^{-1}$-semi-automorphism of $G$,
and $\nu^{-1}\colon Y_2'\to Y_1$ is a $\tau^{-1}$-equivariant $\gamma^{-1}$-semi-isomorphism.
Moreover, $\tau^{-1}(T)=T$ and $\tau^{-1}(B)=B$.
Using \eqref{e:f'2}, \eqref{e:nu-tau}, \eqref{e:b-m-y}, and \eqref{e:veg}, we compute:
\begin{align*}
f_2'(b^{-1}\cdot y_2')&=\gamma(f_1(\nu^{-1}(b^{-1}\cdot y_2')))
= \gamma(f_1(\hs (\tau^{-1}(b))^{-1}\cdot \nu^{-1}(y_2')\hs ))\\
&=\gamma(\chi_1(\tau^{-1}(b)))\cdot \gamma(f_1(\nu^{-1}(y_2')))
=\veg(\chi_1)(b)\cdot f_2'(y_2').
\end{align*}
Thus $f_2'\in K(Y_2')^{(B)}_{\chi_2}$ where $\chi_2=\veg(\chi_1)$, as required.
\end{proof}

\begin{corollary}\label{c:sMY_2'}
$\sX(Y_2')=\veg(\sX(Y_1))$.
\end{corollary}

\begin{proof}
By Lemma \ref{l:Akhiezer-calc} we have $\veg(\sX(Y_1))\subset \sX(Y_2')$.
Applying Lemma \ref{l:Akhiezer-calc} to the triple $(\gamma^{-1},\tau^{-1}, \nu^{-1})$, we obtain that
$\ve_{\gamma^{-1}} ( \sX(Y_2'))\subset\sX(Y_1)$,
hence $ \sX(Y_2')\subset \veg(\sX(Y_1))$.
Thus $\sX(Y_2')=\veg(\sX(Y_1))$, as required.
\end{proof}

Let $v_1\in\Val^B(K(Y_1))$.
We define $\nu_* v_1\in \Val^B(K(Y_2'))$ by
\[(\nu_* v_1)(f_2')=v_1(\nu_*^{-1}(f_2'))\quad\text{for } f_2'\in K(Y_2').\]
We consider the maps
\[ \rho_1\colon \Val^B(K(Y_1))\to V(Y_1)\quad\text{and}\quad \rho_2'\colon \Val^B(K(Y_2'))\to V(Y_2').\]

\begin{lemma}\label{Huruguen-2.18}
For any $v_1\in\Val^B(K(Y_1))$ we have
\[\rho_2'(\nu_* v_1)=\veg(\rho_1(v_1)).\]
\end{lemma}

\begin{proof}
See Huruguen \cite[Proposition 2.18]{Huruguen}.
\end{proof}

\begin{corollary}\label{c:sVY_2'}
$\sV(Y_2')=\veg(\sV(Y_1))$.\hfill \qed
\end{corollary}

Let $D_1\in\sD(Y_1)$ be a color, that is, $D_1$ is the closure of a $B$-orbit of codimension one in $Y_1$.
We set $D_2':=\nu(D_1)\subset Y_2'$; then $D_2'\in\sD(Y_2')$.
We also write $\nu_*D_1$ for $\nu(D_1)$.

Let $D_1\in \sD(Y_1)$ and $D_2'\in\sD(Y_2')$; then we denote by
$\val_1(D_1)\in \Val^B(K(Y_1))$ and $\val_2'(D_2')\in \Val^B(K(Y_2'))$
the corresponding $B$-invariant valuations.

\begin{lemma}\label{Huruguen-2.19}
For any $D_1\in \sD(Y_1)$ we have
\[\val_2'(\nu_*D_1)=\nu_*(\val_1(D_1)).\]
\end{lemma}

\begin{proof}
See Huruguen \cite[Proposition 2.19]{Huruguen}.
 \end{proof}

\begin{remark}
Propositions 2.18 and 2.19 of  Huruguen \cite[Section 2.2]{Huruguen}
were proved in his article under certain additional assumptions.
Namely, Huruguen assumes that $k/k_0$ is a Galois extension,
that the triple $(G,Y,\theta)$ has a $k_0$-model $(G_0,Y_0,\theta_0)$,
and that $Y_0$ has a $k_0$-point $y^{(0)}$.
Those assumptions are not used in his proof.
\end{remark}

By abuse of notation, if $D_1\in \sD(Y_1)$ and $D_2'\in\sD(Y_2')$,
we write $\rho_1(D_1)$ for $\rho_1(\val_1(D_1))\in V(Y_1)$
and  $\rho_2'(D_2')$ for $\rho_2'(\val_2'(D_2'))\in V(Y_2')$.

\begin{corollary}[\rm from Lemma \ref{Huruguen-2.18} and Lemma \ref{Huruguen-2.19}]
\label{c:sDY2'}
For any $D_1\in \sD(Y_1)$ we have
\[ \rho_2'(\nu_*D_1)=\veg(\rho_1(D_1)).\]\qed
\end{corollary}

\begin{lemma}\label{l:D1sDY1}
For any $D_1\in \sD(Y_1)$ we have:
\begin{enumerate}
\item[\rm(i)] $\Stab_G(\nu_*D_1)=\tau(\Stab_G(D_1))$;
\item[\rm(ii)] $\vs(\nu_*D_1)=\veg(\vs(D_1))$.
\end{enumerate}
\end{lemma}

\begin{proof}
(i) follows from the fact that the map $\nu\colon Y_1(k)\to Y_2'(k)$
is $\tau$-equivariant, and (ii) follows from (i).
 \end{proof}

\begin{corollary}[\rm from Corollary \ref{c:sDY2'} and Lemma \ref{l:D1sDY1}]
\label{c:OmegaY2'}
\[ \Omone(Y_2')=\veg(\Omone(Y_1))\quad\text{and}\quad \Omt(Y_2')=\veg(\Omt(Y_1)).\]\qed
\end{corollary}

\begin{proposition}\label{p:Akh}
\begin{align*}
\sX(Y_2)=\veg(\sX(Y_1)\hs ),\quad &\sV(Y_2)=\veg(\sV(Y_1)\hs ),\\
&\Omone(Y_2)=\veg (\Omone(Y_1)\hs ),\quad \Omt(Y_2)=\veg (\Omt(Y_1)\hs ).
\end{align*}
\end{proposition}

\begin{proof}
Since $H_2'$  and $H_2$ are conjugate, by Lemma \ref{l:isom-conj}
the $G$-varieties $Y_2'$ and $Y_2$ are isomorphic,
hence they have the same combinatorial invariants,
and the proposition follows from Corollaries \ref{c:sMY_2'}, \ref{c:sVY_2'}, and \ref{c:OmegaY2'}.
 \end{proof}

Note that Proposition \ref{p:Akh} generalizes Propositions 5.2, 5.3, and 5.4 of Akhiezer \cite {Akhiezer}.
Namely, in the case when $\gamma^2=1$, our Proposition \ref{p:Akh} is equivalent to those results of Akhiezer.
Our proofs are similar to his.

\begin{proposition}\label{p:CF}
With the notation and assumptions of Subsection \ref{ss:H1-H2},
the subgroups $H_1$ and $H_2=\sigma_\gamma(H_1)$ are conjugate
if and only if  $\veg$ preserves the combinatorial invariants of $Y_1$, that is
\begin{equation}\label{e:CF}
\begin{aligned}
\veg(\sX(Y_1)\hs )=\sX(Y_1),\quad &\veg(\sV(Y_1)\hs )=\sV(Y_1),\\
&\veg( \Omone(Y_1)\hs )= \Omone(Y_1),\quad \veg( \Omt(Y_1)\hs )= \Omt(Y_1).
\end{aligned}
\end{equation}
\end{proposition}

Proposition \ref{p:CF} generalizes Theorem 3(1) of Cupit-Foutou \cite{CF},
where the case $k_0=\R$ was considered.

\begin{proof}
If $H_1$ and $H_2$ are conjugate, then by Lemma \ref{l:isom-conj}
the homogeneous spaces $Y_1=G/H_1$ and $Y_2=G/H_2$ are isomorphic as $G$-varieties and  hence,
they have the same combinatorial invariants, that is,
\begin{equation}\label{e:AkhL}
\begin{aligned}
\sX(Y_2)=\sX(Y_1),\quad &\sV(Y_2)=\sV(Y_1),\\
&\Omone(Y_2)=\Omone(Y_1),\quad  \Omt(Y_2)=\Omt(Y_1),
\end{aligned}
\end{equation}
and \eqref{e:CF} follows from Proposition \ref{p:Akh}.
Conversely, if equalities \eqref{e:CF} hold, then
by Proposition \ref{p:Akh} the equalities \eqref{e:AkhL} hold,
and by Proposition  \ref{t:Losev} (Losev's Uniqueness Theorem)
the subgroups $H_1$ and $H_2$ are conjugate.
 \end{proof}

\begin{corollary}\label{c:Akh-CF-bis}
With the notation and assumptions of Subsection \ref{ss:H1-H2},
 there exists a  $\sigma_\gamma$-equivariant $\gamma$-semi-automorphism $\mu\colon  Y_1\to Y_1$,
if and only if $\veg$ preserves the combinatorial invariants of $Y_1$, that is,  equalities \eqref{e:CF} hold.
\end{corollary}

\begin{proof}
By Corollary \ref{c:maincor-3} there exists a $\sigma_\gamma$-equivariant
$\gamma$-semi-automorphism  $\mu$ of $Y_1$
if and only if the subgroup $\sigma_\gamma(H_1)$ of $G$ is conjugate to $H_1$.
Now the corollary follows from Proposition \ref{p:CF}.
 \end{proof}

\section{Equivariant models of automorphism-free\\ spherical homogeneous spaces}
\label{s:aut-free}

\begin{subsec}\label{ss:mt}
Let $k_0$ be a field of characteristic 0 and let $k$ be a fixed algebraic closure of $k_0$
with Galois group $\Gamma=\Gal(k/k_0)$.

Let $G$ be a  connected  reductive group over $k$.
Let $T\subset B\subset G$ be as in Section \ref{s:invariants}.
We consider the based root datum $\BRD(G)=\BRD(G,T,B)$.

Let $G_0$ be a $k_0$-model of $G$.
 For any $\gamma\in\Gamma$, this model defines a $\gamma$-semi-automorphism
\[\sigma_\gamma\colon G\to G,\]
which induces an automorphism $\ve_\gamma\in\Aut\, \BRD(G)$; see Subsection \ref{ss:H1-H2}.
We obtain a homomorphism
\[\ve\colon \Gamma\to\Aut\, \BRD(G),\quad \gamma\mapsto \ve_\gamma\hs .\]

Let $Y=G/H$ be a spherical homogeneous space of $G$.
We consider the combinatorial invariants of $Y$:
\begin{align*}
\sX=\sX(Y)\subset\X^*(T),\quad & \sV=\sV(Y)\subset \Hom_\Z(\sX,\Q),\\
     & \Omone=\Omone(Y),\quad \Omt=\Omt(Y)\subset \Hom_\Z(\sX,\Q)\times\sP(S);
     \end{align*}
see Section \ref{s:invariants}.
Since $\ve_\gamma$ acts on $\BRD(G)$, we can define
\[\ve_\gamma(\sX),\ \ve_\gamma(\sV),\ \ve_\gamma(\Omone),\ \ve_\gamma(\Omt). \]
\end{subsec}

Recall that $Y=G/H$.
By Lemma \ref{l:nu-1-2}(i) we have $\gamma_*Y\cong G/\sigma_\gamma(H)$.

\begin{proposition}\label{p:Huruguen}
If $Y=G/H$ admits a  $G_0$-equivariant $k_0$-model $Y_0$,
then for all $\gamma\in \Gamma$, the automorphism $\veg$ preserves the combinatorial invariants of $Y$, that is
\begin{equation}\label{e:Gamma-stable}
\ve_\gamma(\sX)=\sX,\ \ve_\gamma(\sV)=\sV,\ \ve_\gamma(\Omone)= \Omone,\ \ve_\gamma(\Omt)=\Omt.
\end{equation}
\end{proposition}

Proposition \ref{p:Huruguen} follows from formulas of Huruguen \cite[Section 2.2]{Huruguen}.
For the reader's convenience we prove it here.

\begin{proof}
A $G_0$-equivariant $k_0$-model $Y_0$ of $Y$ defines, for any $\gamma\in \Gamma$,
a $\sigma_\gamma$-equivariant $\gamma$-semi-automorphism $\mu_\gamma$ of $Y$,
hence an isomorphism of $G$-$k$-varieties
$$(\mu_\gamma)_\natural\colon \  G/\sigma_\gamma(H)\cong\gamma_*Y\,\labelto{\sim} Y.$$
We see that the $G$-varieties $G/H$ and $G/\sigma_\gamma(H)$ are isomorphic,
hence they have the same combinatorial invariants.
By Proposition  \ref{p:Akh} the combinatorial invariants
of the spherical homogeneous space $G/\sigma_\gamma(H)$ are
\[\ve_\gamma(\sX),\ \ve_\gamma(\sV),\ \ve_\gamma(\Omone),\ \ve_\gamma(\Omt),\]
and \eqref{e:Gamma-stable} follows.
 \end{proof}

The next theorem is a partial converse of Proposition \ref{p:Huruguen}.

\begin{theorem}\label{t:main}
Let $k,\ k_0,\ \Gamma,\ G,\ H,\ G_0$ be as in \ref{ss:mt}.
Assume that:
\begin{enumerate}
\item[\rm(i)] For all $\gamma\in \Gamma$,  \,$\veg$ preserves the combinatorial invariants of $Y=G/H$,
that is, equalities \eqref{e:Gamma-stable} hold, and
\item[\rm(ii)] $\NGH=H$.

\end{enumerate}
Then $Y$ admits a  $G_0$-equivariant  $k_0$-model $Y_0$.
This $k_0$-model is unique up to a unique isomorphism.
\end{theorem}

\begin{proof}
Let $\gamma\in \Gamma$.
Since  $\veg$ preserves the combinatorial invariants of $Y$,
by Corollary \ref{c:Akh-CF-bis} there exists a {\em $\sigma_\gamma$-equivariant}  $\gamma$-semi-automorphism
\[\mu_\gamma\colon  Y\to Y.\]
Thus condition (iv) of Lemma \ref{l:model} is satisfied.

Since $\NGH=H$, by Corollary \ref{c:NGH}  we have $\Aut^G(Y)=\{1\}$,
and by Lemma \ref{l:Aut-1} conditions (i) and (ii) of Lemma \ref{l:model}
are satisfied.

The variety $Y=G/H$ is quasi-projective (see Springer \cite[Corollary 5.5.6]{Springer}\hs)
and hence, condition (iii) of Lemma \ref{l:model} is satisfied.

By  Lemma \ref{l:model} there exists a  $G_0$-equivariant  $k_0$-model $Y_0$ of $Y$.
Since  $\Aut^G(Y)=\{1\}$, for any given $\gamma\in\Gamma$ the $\gamma$-semi-automorphism $\mu_\gamma$ is unique,
hence the model $Y_0$ is unique up to a unique isomorphism.
 \end{proof}

Recall that a $k_0$-model $G_0$ of a  connected  reductive $k$-group $G$ is called an {\em inner form}
(of a split group) if $\ve_\gamma=1$ for all $\gamma\in\Gamma = \Gal(k/k_0)$.

\begin{lemma}\label{l:inner-cond}
Let $k,\ k_0,\ \Gamma,\ G,\ H,\ G_0$ be as in \ref{ss:mt}.
Then each of the conditions below imply condition (i) of Theorem \ref{t:main}.
\begin{enumerate}
\item[\rm(i)] $G_0$ is an inner form;
\item[\rm(ii)] $G_0$ is absolutely simple (that is, $G$ is simple)
of one of the types $\AAA_1$, $\BB_n$, $\CC_n$, $\EE_7$, $\EE_8$, $\FF_4$, $\GG_2$;
\item[\rm(iii)] $G_0$ is split.
\end{enumerate}
\end{lemma}

\begin{proof}
(i) If $G_0$ is an inner form, then $\ve_\gamma=1$ for every $\gamma\in\Gamma$,
hence condition (i) of Theorem \ref{t:main} is clearly satisfied.

(ii) In these cases the Dynkin diagram  $\Dyn(G)$ has no nontrivial automorphisms,
hence $\Gamma$ acts trivially on $\Dyn(G)$ and on $\BRD(G)$.
We see that (ii) implies (i).

(iii) In this case clearly $G_0$ is an inner form.
 \end{proof}

\begin{corollary}
If  $\NGH=H$ and at least one of the conditions (i--iii)
of Lemma \ref{l:inner-cond} is satisfied,
then $Y$ admits a  $G_0$-equivariant $k_0$-model, and this $k_0$-model is unique. \qed
\end{corollary}

\begin{remark}
Assume that $k=\R$ and $\NGH=H$.
The assertion that if condition (iii) of Lemma \ref{l:inner-cond}
is satisfied, then $Y$ has a unique  $G_0$-equivariant  $\R$-model $Y_0$,
is Theorem 4.12 of Akhiezer and Cupit-Foutou \cite{ACF}.
The similar assertion when only condition (i) of Lemma \ref{l:inner-cond}
is satisfied, is Theorem 1.1 of Akhiezer \cite{Akhiezer}.
Our article is inspired by this result of Akhiezer,
and our proof of Theorem \ref{t:main} is similar to his proof.
\end{remark}

\section{Equivariant models of spherically closed spherical homogeneous spaces}
\label{s:sph-closed}

In this section we do not assume that $\NGH=H$.

\begin{subsec}
Let $k,\ G,\  H,\ Y=G/H,\ T\subset B\subset G$ be as in Section \ref{s:invariants},
in particular, $k$ is an algebraically closed
field of characteristic 0.
The group $\Aut^G(Y)=\NGH/H$ acts on $Y$
and on the set $\sD$ of colors of $Y$.

Recall that a spherical homogeneous space $Y=G/H$ is called {\em spherically closed}
if $\NGH/H$ acts on $\sD$ {\em faithfully}, that is, if the homomorphism
\[ \Aut^G(Y)\to \Aut(\sD)\]
is injective. (Here $\Aut(\sD)$ denotes the group of permutations of the finite set $\sD$.)

Let $k_0\subset k$ be a subfield such that $k$ is an algebraic closure of $k_0$.
Let $G_0$ be a $k_0$-model of $G$,
and for $\gamma\in\Gamma:=\Gal(k/k_0)$
let   $\sigma_\gamma\colon G\to G$ be the $\gamma$-semi-automorphism defined by $G_0$.
Let $\ve_\gamma\colon \X^*(T)\to \X^*(T)$ be as in \eqref{e:veg}.
\end{subsec}

\begin{theorem} \label{t:CF-CF}
Let  $k$ be an algebraically closed field of characteristic 0.
Let $G$, $H$, $Y=G/H$ be as in Section  \ref{s:invariants}.
Let  $G_0$ be a $k_0$-model of $G$, where $k_0\subset k$ is a subfield
such that $k$ is an algebraic closure of $k_0$.
Assume that
\begin{enumerate}
\item[\rm(i)] $\veg$ preserves the combinatorial invariants of $Y$ for all $\gamma\in\Gamma$, and
\item[\rm(ii)] $Y$ is spherically closed.
\end{enumerate}
Then $Y$ admits a  $G_0$-equivariant  $k_0$-model.
\end{theorem}

Theorem \ref{t:CF-CF} generalizes the existence assertion of Theorem \ref{t:main}.
It was inspired by Corollary 1 of Cupit-Foutou \cite[Section 2.5]{CF}, where the case $k_0=\R$ was considered.

In order to prove the theorem we need a few lemmas.

\begin{lemma}\label{l:perm-cover}
Let $\zeta\colon \sD\to \Omega$ be a mapping of nonempty finite sets.
Let $\Gamma$ be a group acting on $\Omega$ by a homomorphism
\[s\colon\Gamma\to \Aut(\Omega),\quad \gamma\mapsto s_\gamma.\]
Assume that for every $\gamma\in\Gamma$ there exists
a permutation $m_\gamma\colon \sD\to \sD$ covering $s_\gamma$,
that is, such that the following diagram commutes:
\[
\xymatrix{
\sD\ar[r]^{m_\gamma} \ar[d]_\zeta  &\sD\ar[d]^\zeta \\
\Omega\ar[r]^{s_\gamma}                 & \Omega
}
\]
Then there exists a {\emm homomorphism} $m'\colon\Gamma\to\Aut(\sD)$ such that:
\begin{enumerate}
\item[\rm(i)] for every $\gamma\in\Gamma$ the permutation  $m'_\gamma$ covers $s_\gamma$;
\item[\rm(ii)] for every $\gamma\in\Gamma$ we have $m'_\gamma=a_\gamma\circ m_\gamma$,
where $a_\gamma\in\Aut_\Omega(\sD)$;
\item[\rm(iii)]  $m'_\gamma=\id_\sD$ for all $\gamma\in\ker\hs s$.
\end{enumerate}
\end{lemma}

\begin{proof}
We may and shall assume that $\Gamma$ acts transitively on $\Omega$.
Let $\omega,\omega' \in \Omega$; then there exists $\gamma\in \Gamma$
such that $s_\gamma(\omega)=\omega'$.
By hypotheses there exists $m_\gamma\in\Aut(\sD)$ covering $s_\gamma$\hs.
Then $m_\gamma$ induces a bijection
$\zeta^{-1}(\omega)\to\zeta^{-1}(\omega')$, hence the cardinalities
of $\zeta^{-1}(\omega)$ and  $\zeta^{-1}(\omega')$ are equal.
We see that $\omega\mapsto |\zeta^{-1}(\omega)|$ is a constant function on $\Omega$; we denote its value by $n$.
For each $\omega\in \Omega$ we fix some bijection between $\zeta^{-1}(\omega)$ and the set $\{1,\dots,n\}$;
we denote the element of $\zeta^{-1}(\omega)\subset\sD$ corresponding to $i\in \{1,\dots,n\}$
by $d_\omega^{(i)}$.
Then we {\em construct} $m'_\gamma\in\Aut(\sD)$ by setting
\[ m'_\gamma(\hs d_\omega^{(i)}\hs )=d_{s_\gamma(\omega)}^{(i)}.\]
Since $s\colon\gamma\mapsto s_\gamma$ is a homomorphism, we see that $m'\colon\gamma\mapsto m'_\gamma$
is a homomorphism, and clearly $m'_\gamma$ covers $s_\gamma$, which proves (i).
Set $a_\gamma=m'_\gamma\circ m_\gamma^{-1}$; then clearly (ii) holds, and the assertion (iii) holds
by construction.
 \end{proof}

\begin{subsec}\label{ss:zeta-Omega-s-gamma}
Write
\[ \zeta=\rho\times\vs\colon\  \sD\to V\times\sP(S),\quad \Omega=\im\, \zeta,\quad
           s_\gamma=\ve_\gamma|_\Omega\colon \,\Omega\to \Omega;\]
then the map
\[\Gamma\to \Aut(\Omega),\quad \gamma\mapsto s_\gamma\]
is a homomorphism.
Assume that for all $\gamma\in \Gamma$ there exists a $\sigma_\gamma$-equivariant
$\gamma$-semi-automorphism $\mu\colon Y\to Y$, that is,
a $\gamma$-semi-automorphism satisfying
\begin{equation}\label{e:s-g-eq}
\mu(g\cdot y)=\sigma_\gamma(g)\cdot \mu(y)\quad\text{for all }g\in G(k),\ y\in Y(k).
\end{equation}
The following lemma is obvious:
\end{subsec}

\begin{lemma}\label{l:gam-sig-mu}
If $\gamma,\delta\in\Aut(k/k_0)$, $\mu$ is a $\sigma_\gamma$-equivariant
$\gamma$-semi-automorphism, and $\nu$ is  a $\sigma_\delta$-equivariant $\delta$-semi-automorphism,
then $\mu\nu$ is a $\sigma_{\gamma\delta}$-equivariant $\gamma\delta$-semi-automorphism
and $\mu^{-1}$  is a  $\sigma_{\gamma^{-1}}$-equivariant $\gamma^{-1}$-semi-automorphism. \qed
\end{lemma}

\begin{subsec}\label{ss:m-mu}
Consider $\sigma_\gamma(T)\subset\sigma_\gamma(B)\subset G$.
There exists $g_\gamma\in G(k)$ such that if we set $\sigma_\gamma'=\inn(g_\sigma)\circ \sigma_\gamma$,
then
\begin{equation}\label{e:g-gamma}
\sigma_\gamma'(T)=T\quad\text{and}\quad \sigma_\gamma'(B)=B.
\end{equation}
Let $\mu\colon Y\to Y$ be  as in \eqref{e:s-g-eq}.
We define a $\gamma$-semi-automorphism
\[\mu'=g_\gamma\circ\mu\colon Y\to Y,\quad y\mapsto g_\gamma\cdot \mu(y)\quad \text{for }y\in Y(k).\]
Then for $g\in G(k),\ y\in Y(k)$, we have
\begin{equation}\label{e:sigma'-mu'}
\begin{aligned}
\mu'(g\cdot y)=g_\gamma\cdot\mu(g\cdot y)&=g_\gamma\cdot \sigma_\gamma(g)\cdot\mu(y)\\
&=(g_\gamma\cdot\sigma_\gamma(g)\cdot g_\gamma^{-1})\cdot(g_\gamma\cdot \mu(y)) =\sigma_\gamma'(g)\cdot \mu'(y).
\end{aligned}
\end{equation}

Let $D\in \sD=\sD(Y)$ be a color, this means that $D$ is the closure of a codimension one $B$-orbit in $Y$.
Since $\sigma_\gamma'(B)=B$, it follows from \eqref{e:sigma'-mu'} that the divisor $\mu'(D)$ in $Y$
is the closure of a codimension one $B$-orbit, that is, a color.
We obtain a permutation
\begin{equation}\label{e:mu-m}
m_\mu\colon \sD\to \sD,\quad D\mapsto \mu'(D),
\end{equation}
covering $s_\gamma$.
Since $g_\gamma$ for which \eqref{e:g-gamma} holds
is defined uniquely up to multiplication on the left by an element $t\in T(k)\subset B(k)$,
we see that $m_\mu$ depends only on $\mu$ and does not depend on the choice of $g_\gamma$.
\end{subsec}

\begin{lemma}\label{l:homomorphism-m}
The map $\mu\mapsto m_\mu$ is a homomorphism: for $\gamma,\mu,\delta,\nu$ as in Lemma \ref{l:gam-sig-mu},
we have $m_{\mu\circ\nu}=m_\mu\circ m_\nu$.
\end{lemma}

\begin{proof} A straightforward calculation. \end{proof}

\begin{subsec}  {\em Proof of Theorem {\ref{t:CF-CF}}.}
Let $\gamma\in\Gamma$.
Since  $\veg$ preserves the combinatorial invariants of $Y=G/H$,
by Corollary \ref{c:Akh-CF-bis} there exists a $\sigma_\gamma$-equivariant $\gamma$-semi-automorphism
$\mu_\gamma\colon Y\to Y$.
Set
\[ m_\gamma=m_{\mu_\gamma}\in\Aut(\sD), \]
see Subsection \ref{ss:m-mu}; then $m_\gamma$ covers $s_\gamma$,
where $s_\gamma\in\Aut(\Omega)$ is the restriction of $\veg$ to $\Omega$.
By Lemma \ref{l:perm-cover} there exists a {\em homomorphism}
\[ m' \colon \Gamma \to\Aut(\sD),\quad \gamma\mapsto m'_\gamma\]
such that for every $\gamma\in\Gamma$ the permutation $m'_\gamma\in\Aut(\sD)$
covers $s_\gamma$ (property (i)\hs) and
 we have $m'_\gamma=a_\gamma\circ m_\gamma$, where $a_\gamma\in\Aut_\Omega(\sD)$ (property (ii)\hs).
By Theorem \ref{t:Losev-3} there exists an automorphism
$\atil_\gamma\in \Aut^G(Y)$ inducing  $a_\gamma$ on $\sD$.
We set
\[\mu'_\gamma=\atil_\gamma\circ\mu_\gamma\hs.\]
Then $\mu'_\gamma$ is a $\sigma_\gamma$-equivariant
 $\gamma$-semi-automorphism of $Y$ and hence, condition (iv) of  Lemma \ref{l:model} is satisfied.
By Lemma \ref{l:homomorphism-m}, $\mu'_\gamma$ acts on $\sD$ by $a_\gamma\circ m_\gamma=m'_\gamma$.

Let $\gamma,\delta\in \Gamma$; then $\mu'_\gamma\hs \mu'_\delta\hs (\mu'_{\gamma\delta})^{-1}\in\Aut^G(Y)$
and by Lemma \ref{l:homomorphism-m} it acts on $\sD$ by
 $m'_\gamma\hs m'_\delta\hs (m'_{\gamma\delta})^{-1}=\id_\sD$.
Since $Y$ is spherically closed,
we conclude that $\mu'_\gamma\hs \mu'_\delta\hs (\mu'_{\gamma\delta})^{-1}=\id_Y$,
hence $\mu'_{\gamma\delta}=\mu'_\gamma\circ\mu'_\delta$.
Thus the map $\gamma\mapsto\mu'_\gamma$ is a homomorphism and hence,
condition (i) of   Lemma \ref{l:model} is satisfied.

By Corollary \ref{c:anon-anon}, the $G$-variety $Y$ admits a  $G_{k_2}$-equivariant $k_2$-model $Y_2$
over some finite Galois extension $k_2/k_0$ in $k$.
Let $\Gamma_2=\Gal(k/k_2)$, and for $\gamma\in \Gamma_2$ let $\mu''_\gamma$
denote the $\gamma$-semi-automorphism of $Y$ defined by the
$k_2$-model $Y_2$.
After passing to a finite extension, we may assume that for $\gamma\in \Gamma_2$
we have $s_\gamma=\id_\Omega$, and by property (iii) of Lemma \ref{l:perm-cover} we have $m'_\gamma=\id_\sD$\hs .
Moreover, we may assume that for $\gamma\in \Gamma_2$ the semi-automorphism  $\mu''_\gamma$ acts trivially on $\sD$.
It follows that $(\mu''_\gamma)^{-1}\mu'_\gamma$ acts trivially on $\sD$,
and clearly  $(\mu''_\gamma)^{-1}\mu'_\gamma\in \Aut^G(Y)$.
Since $Y$ is spherically closed, we conclude that $(\mu''_\gamma)^{-1} \mu'_\gamma=\id_Y$,
and hence, $\mu'_\gamma= \mu''_\gamma$ for $\gamma\in\Gamma_2$.
We see that the homomorphism $\gamma\mapsto\mu'_\gamma$ satisfies condition (ii) of Lemma \ref{l:model}.
Note that  $Y=G/H$ is quasi-projective, that is, condition (iii) of Lemma \ref{l:model} is satisfied as well.
By Lemma \ref{l:model} the homomorphism $\gamma\mapsto\mu'_\gamma$
defines a  $G_0$-equivariant $k_0$-model $Y_0$ of $Y$,
which completes the proof of the theorem. \qed
\end{subsec}

\begin{remark}\label{r:anti-CF}
Let $k=\C$, $k_0=\R$, $\Gamma=\Gal(\C/\R)=\{1,\gamma\}$.
In this case Theorem  \ref{t:CF-CF} means that if $\ve_\gamma$
corresponding to a real structure $\sigma_\gamma\colon G\to G$
preserves the combinatorial invariants of  $Y=G/H$, and $Y$ is spherically closed, then
{\em there exists} an anti-holomorphic $\sigma_\gamma$-equivariant semi-automorphism
$\mu_\gamma\colon Y\to Y$ such that $\mu_\gamma^2=1$.
Note that in general it is not true that
then {\em for every}  anti-holomorphic $\sigma_\gamma$-equivariant semi-automorphism
$\mu_\gamma\colon Y\to Y$  we have $\mu_\gamma^2=1$; see Example \ref{ex:anti-CF} below.
\end{remark}

\begin{example}\label{ex:anti-CF}
Let
\begin{gather*}
G_1=G_2=\SO_{3,\C}\hs,\ H_1=H_2=\SO_{2,\C}\hs,\ Y_1=G_1/H_1\hs,\ Y_2=G_2/H_2\hs,\\
G=G_1\times_\C G_2\hs,\ Y=Y_1\times_\C Y_2\hs.
\end{gather*}
Then $Y$ is a spherically closed spherical homogeneous space of $G$.
Let $\sD_1=\{D_1^+, D_1^-\}$ denote the set of colors of $Y_1$, and let $\sD_2=\{D_2^+, D_2^-\}$
denote the set of colors of $Y_2$;
then  the set of colors $\sD$ of $Y$ can be identified with $\sD_1\sqcup\sD_2$ (disjoint union).
Let $G_0=R_{\C/\R} \SO_{3,\C}$ (the Weil restriction of scalars from $\C$ to $\R$),
which is a real model of $G$, and let $H_0=R_{\C/\R} \SO_{2,\C}\subset G_0$.
We set $Y_0=G_0/H_0=R_{\C/\R} Y$.
Let $\sigma_\gamma\colon G\to G$ and $\mu_\gamma\colon Y\to Y$
correspond to the real models $G_0$ and $Y_0$, respectively.
Then $\mu_\gamma$ acts on $\sD$ by the permutation of order 2
\[ m_\mu=(D_1^+, D_2^+)\cdot(D_1^-,D_2^-)\ \in\Aut(\sD),\]
where $(D_1^+, D_2^+)$ and $(D_1^-,D_2^-)$ are transpositions.
Consider the transposition
\[ a=(D_1^+,D_1^-)\in\Aut (\sD);\]
then the permutation $a$ is induced by the nontrivial element
\[\atil\in {\rm O}_{2,\C}/\SO_{2,\C}=\Aut^{G_1}(Y_1)\subset \Aut^G(Y).\]
Set
\[\mu'_\gamma =\atil\circ\mu_\gamma\colon \ Y\to Y,\]
which is an anti-holomorphic $\sigma_\gamma$-equivariant semi-automorphism of $Y$.
Since $\mu'_\gamma$ acts on $\sD$ as $a\circ m_\gamma$, which is a permutation of order 4,
we conclude that $(\mu'_\gamma)^2\neq 1$.
\end{example}

\begin{subsec}\label{ss:obvious}
In Example \ref{ex:CF} we considered a spherically closed spherical variety $Y=G/H$,
where $G=\SL_{2,k}$ and $H=T$, a maximal torus in $G$.
In this case it is obvious that for any $k_0$-model $G_0$ of $G$
there exists a $G_0$-equivariant $k_0$-model $Y_0$ of $Y=G/H$.
Indeed, there exists a maximal torus $T_0\subset G_0$ defined over $k_0$,
and it is clear that $Y_0:=G_0/T_0$ is a $G_0$-equivariant $k_0$-model  of $Y=G/T$.
In the following example we consider a spherically closed spherical subgroup
that is not conjugate to a subgroup defined over $k_0$.
\end{subsec}

\begin{example} \label{ex:Avdeev}
Let $k=\C$, $k_0=\R$. Following a suggestion of Roman Avdeev, we take  $G=\SO_{2n+1,\C}$, where $n\ge 2$,
and we take for $H$ a Borel subgroup of $\SO_{2n,\C}$, where $\SO_{2n,\C}\subset \SO_{2n+1,\C}=G$.
By Proposition \ref{p:Avdeev-bis} below,
$H$ is a spherically closed spherical subgroup of $G$ and $\sN_G(H)\neq H$.
Take $G_0=\SO_{2n+1,\R}$; then $G_0$ is an anisotropic (compact) $\R$-model of $G$.
Since the Dynkin diagram $\BB_n$ of $G$ has no nontrivial automorphisms, $G_0$ is an inner form.
We wish to show that $Y=G/H$ admits a $G_0$-equivariant $\R$-model.
The subgroup $H$ is not conjugate to any subgroup $H_0$ of  $G_0$ defined over $\R$
because $H$ is not reductive; see Lemma \ref{l:Borel-Tits} below.
We see that we cannot argue as in Subsection \ref{ss:obvious}.
Since $\sN_G(H)\neq H$, we cannot apply Theorem \ref{t:main} either.
However, since $H$ is spherically closed,
by Theorem \ref{t:CF-CF} the homogeneous variety $Y=G/H$ does admit a $G_0$-equivariant $\R$-model $Y_0$.
\end{example}

\begin{proposition}[\rm Roman Avdeev, private communication]
\label{p:Avdeev-bis}
Let $G=\SO_{2n+1,\C}$, where $n\ge2$.
Let $H$ be a Borel subgroup of \hs$\SO_{2n,\C}$, where $\SO_{2n,\C}\subset \SO_{2n+1,\C}=G$.
Then  $H$ is a spherically closed spherical subgroup of $G$, and $\sN_G(H)\neq H$.
\end{proposition}

\begin{proof}
Write $\Ggg=\Lie(G)$.
Choose a Borel subgroup $B \subset G$ and a maximal torus $T \subset B$.
Let $X=\X^*(T)$ denote the character lattice of~$T$ and let $R =R(G,T)\subset X$ be the root system.
The Borel subgroup $B$ defines a set of positive roots $R^+\subset R$
and the corresponding set of simple roots $S\subset R^+\subset R$.
Let $U$ denote the unipotent radical of~$B$ and put $\uu = \Lie(U)$.
We have
\[\Ggg=\Lie(T)\oplus\bigoplus_{\beta\in R} \Ggg_\beta\hs,\qquad \uu=\bigoplus_{\beta\in R^+} \Ggg_\beta\hs,\]
where $\Ggg_\beta$ is the root subspace corresponding to a root $\beta$.

Let $R_l\subset R$ denote the root subsystem consisting of the {\em long} roots.
Observe that $R$ is a root system of type $\BB_n$, and $R_l$ is a root system of type $\DD_n$.
Set $R^+_l=R^+\cap R_l$. We set
\[\Ggg_l=\Lie(T)\oplus\bigoplus_{\beta\in R_l} \Ggg_\beta\hs,\qquad \uu_l=\bigoplus_{\beta\in R^+_l} \Ggg_\beta\hs,\]
where the direct sums are taken over  long roots.
Let $G_l$  (resp., $U_l$) be the connected algebraic subgroup of $G$ with Lie algebra $\Ggg_l$ (resp., $\uu_l$).
Set $H=TU_l$.
Then $G_l\simeq\SO_{2n,\C}$ and $H$ is a Borel subgroup of $G_l$.

It is well known that $H$ is a spherical subgroup of~$G$.
For example, this fact follows from Theorem 1 of Avdeev \cite{Av11} (to apply this theorem
one needs to check that the short positive roots in~$R$ are linearly independent).
By Avdeev \cite[Proposition~5.25]{Avdeev} $H$ is spherically closed.

We consider the Weyl group $W=W(G,T)=W(R)$.
Let $r\in W=\sN_G(T)/T$ denote the reflection with respect to the {\em short} simple root,
and let $\rho$ be a  representative of $r$  in $\sN_G(T)$.
Since $r$ preserves $R^+_l$, we see that $\rho \in \sN_G(H)$.
Since $\rho\notin T$ and $\sN_G(T)\cap B=T$, we see that $\rho\notin B$.
By construction $H\subset B$, and we conclude that $\rho\notin H$, hence $\sN_G(H) \neq H$.
In fact, $\sN_G(H) = H \cup \rho H$ by Avdeev \cite[Theorem~3]{Av13}.
 \end{proof}

\begin{lemma}[\rm well-known]
\label{l:Borel-Tits}
Let $k_0$ be a field of characteristic 0, and let $G_0$
be a connected, reductive, anisotropic $k_0$-group.
Then any connected $k_0$-subgroup $H_0\subset G_0$ is reductive.
\end{lemma}

\begin{proof}
For the sake of contradiction, assume that $H_0$ is not reductive.
Let $U_0$ denote the unipotent radical of $H_0$, which is a nontrivial unipotent $k_0$-subgroup.
Let $u\in U_0(k_0)\subset H_0(k_0)\subset G_0(k_0)$ be a non-unit element of $U_0(k_0)$.
We see that $G_0(k_0)$ contains a non-unit nilpotent element.
On the other hand, by Borel--Tits \cite[Corollary 8.5]{Borel-Tits}, all $k_0$-points
of a connected, reductive, anisotropic $k_0$-group are semisimple elements.
Contradiction.
 \end{proof}

The following example shows that $G/H$ might have no $G_0$-equivariant $k_0$-model
when $H$ is not spherically closed.

\begin{example}\label{ex:c-ex}
Let $k=\C$, $k_0=\R$.
Let  $n\ge 1$,  $G=\Sp_{2n,\C} \times_\C \Sp_{2n,\C}$, $Y=\Sp_{2n,\C}$, the group
$G$ acts on $Y$ by
\[ (g_1,g_2)*y=g_1\hs y\hs g_2^{-1},\quad g_1,\hs g_2,\hs y\in\Sp_{2n}(\C).\]
Let $H$ denote the stabilizer in $G$ of $1 \in\Sp_{2n}(\C)=Y(\C)$;
then  $H=\Sp_{2n,\C}$ embedded diagonally in $G$.
We have  $Y=G/H$,
and $Y$ is a spherical homogeneous space of $G$.
We have $\NGH=Z(G)\!\cdot\! H$, where $Z(G)$ denotes the center of $G$.
It follows that  $\NGH/H\simeq\{\pm 1\}\neq \{1\}$.
Clearly $\NGH/H$ acts trivially on $\sD(G/H)$, so $H$ is not spherically closed.

Consider the following  real model of $G$:
\[ G_0=\Sp_{2n,\R}\times_\R \Sp(n),\]
where $\Sp(n)$ is the compact real form of $\Sp_{2n}$.
We show that $Y$ cannot have a $G_0$-equivariant real model,
although $G_0$ is an inner form of a split group.

Indeed, assume for the sake of contradiction
that such a real model $Y_0$ of $Y$ exists.
We have $Y=\Sp_{2n,\C}$, and $Y_0$  is simultaneously a principal homogeneous
space of $\Sp_{2n,\R}$  and of $\Sp(n)$.
Since $H^1(\R, \Sp_{2n,\R})=1$ (see  Serre \cite[III.1.2, Proposition 3]{Serre}\hs),
we see that $Y_0(\R)$ is not empty.
It follows that the topological space $Y_0(\R)$ is simultaneously a principal
homogeneous space of $\Sp(2n,\R)$ and of $\Sp(n)$.
Thus $Y_0(\R)$ is simultaneously homeomorphic to
the noncompact Lie group $\Sp(2n,\R)$ and to the compact Lie group $\Sp(n)$, which is clearly impossible.
Thus, there is  no  $G_0$-equivariant  real model $Y_0$ of $Y$.
\end{example}

\begin{subsec}\label{ss:Boris-Boris}
Let $k,\ k_0,\ \Gamma,\ G,\ H,\ G_0$ be as in Subsection \ref{ss:mt}, in particular,
$k$ is an algebraically closed field of characteristic 0 and $\Gamma=\Gal(k/k_0)$.
We assume that $H$ is spherically closed and that $Y=G/H$ admits a $G_0$-equivariant $k_0$-model $Y_0$.
Then by Corollary \ref{c:Akh-CF-bis}, \,$\veg$ preserves
the combinatorial invariants of $Y$ for all $\gamma\in\Gamma$,
in particular, $\Gamma$ acts on the finite set $\Omt=\Omt(Y)$.
Let $U_1,U_2,\dots,U_r$ be the orbits of $\Gamma$ in $\Omt$.
For each $i=1,2,\dots,r$, let us choose a point $u_i\in U_i$.
Set $\Gamma_i=\Stab_\Gamma(u_i)$.
\end{subsec}

\begin{theorem}\label{t:Boris-Boris}
With the notation and assumptions of \ref{ss:Boris-Boris} we have:
\begin{enumerate}
\item[\rm(i)] The set of isomorphism classes of $G_0$-equivariant $k_0$-models of $Y$
is canonically a principal homogeneous space of the abelian group
$H^1(\Gamma,\Aut_\Omega(\sD))$;
\item[\rm(ii)] $H^1(\Gamma,\Aut_\Omega(\sD))\simeq\prod_{i=1}^r \Hom(\Gamma_i,\St)$,
where $S_2$ is the symmetric group on two symbols.
\end{enumerate}
\end{theorem}

\begin{corollary}\label{c:Boris-Boris}
In Theorem \ref{t:Boris-Boris} assume that $|\Gamma|=2$.
Then the number of isomorphism classes of $G_0$-equivariant $k_0$-models of $Y=G/H$
is $2^s$, where $s$ is the number of fixed points of $\Gamma$ in $\Omt$.
\end{corollary}

\begin{proof}
Let $U_i$ be an orbit of $\Gamma$ in $\Omt$.
If $|U_i|=2$, then $\Gamma_i=\{1\}$, hence $|\Hom(\Gamma_i,\St)|=1$.
If $|U_i|=1$, then $\Gamma_i=\Gamma$, and hence, $|\Hom(\Gamma_i,\St)|=2$.
Now the corollary follows from the theorem.
 \end{proof}

\begin{proof}[Proof of Theorem \ref{t:Boris-Boris}]
Since $Y$ is quasi-projective,  there is a canonical bijection
between the set of the isomorphism classes in the theorem and the pointed set $H^1(\Gamma, \Aut^G(Y))$;
see Serre  \cite{Serre}, Proposition 5 in Section III.1.3.
By Theorem 2 of Losev \cite{Losev} (see also Subsection \ref{ss:Losev-Aut} below),
the group $\Aut^G(Y)$ is abelian,
hence $H^1(\Gamma, \Aut^G(Y))$ is an abelian group.
It follows that the set of isomorphism classes in the theorem
is canonically a principal homogeneous space of this abelian group;
see Serre \cite{Serre}, Proposition 35 bis and Remark (1) thereafter in Section I.5.3.
By Corollary \ref{c:closed-isom} there is a canonical isomorphism of abelian groups
$\Aut^G(Y)\isoto\Aut_\Omega(\sD)$, and (i) follows.

We compute $H^1(\Gamma,\Aut_\Omega(\sD))$.
Recall that we have a surjective map $\zeta\colon \sD\to \Omega$.
Set $\sD^{(2)}=\zeta^{-1}(\Omt)$; then clearly
\[\Aut_\Omega(\sD)=\Aut_\Omt(\sD^{(2)})=\prod_{\omega\in\Omt}\St=\prod_{i=1}^r \left(\prod_{\omega\in U_i} \St\right),\]
hence
\[ H^1(\Gamma,\Aut_\Omega(\sD))=\prod_{i=1}^r H^1\left(\Gamma, \prod_{\omega\in U_i} \St\right).\]
Since $\Gamma$ acts on $U_i$ transitively, the $\Gamma$-module $\prod_{\omega\in U_i} \St$
is induced by the $\Gamma_i$-module $\St$ (see Serre \cite[Section I.2.5]{Serre}) and hence,
by the lemma of Faddeev and Shapiro,
see Serre \cite[I.2.5, Proposition 10]{Serre}, we have
\[H^1\left(\Gamma, \prod_{\omega\in U_i} \St\right)\simeq H^1(\Gamma_i,\St)=\Hom(\Gamma_i,\St).\]
Thus
\[H^1(\Gamma,\Aut_\Omega(\sD))\simeq\prod_{i=1}^r \Hom(\Gamma_i,\St),\]
which proves (ii).
 \end{proof}

\begin{example}\label{ex:SO3}
Let $G=\SO_{3,\C}\simeq\PGL_{2,\C}$.
Consider the maximal torus $H=\SO_{2,\C}\subset\SO_{3,\C}=G$.
Consider the affine quadric $Y$ in $\A_\C^3$ given by the equation
\[  x_1^2+x_2^2+x_3^2=1.\]
The group $G=\SO_{3,\C}\subset \GL(3,\C)$ naturally acts in $\A^3_\C$ and preserves $Y$.
The stabilizer in $G$  of the point  $(0,0,1)\in Y(\C)$ is $H=\SO_{2,\C}$, and therefore,
we may identify $Y$ with $G/H$.
The homogeneous space  $Y=G/H$ is a spherically closed spherical homogeneous space of $G$; see Example \ref{ex:CF}.
We have
\[\Aut^G(G/H)=\sN_G(H)/H={\rm O}_{2,\C}/\SO_{2,\C}\simeq\{\pm 1\}.\]

Let $G_0$ be an $\R$-form of $G$; then there exists a maximal torus  $H_0$ in $G_0$ (defined over $\R$),
and  clearly $G_0/H_0$ is a $G_0$-equivariant $\R$-model of $Y=G/H$
(so we do not have to refer to Theorem \ref{t:CF-CF}
in order to see that $Y$ admits a $G_0$-equivariant real model).
Since
\[ H^1(\Gamma,\Aut^G(G/H))=\Hom(\Gamma,\{\pm1\})\simeq\{\pm1\},\]
the variety $Y$ has exactly two $G_0$-equivariant $\R$-models.
We describe these models for each $\R$-model of $G=\SO_{3,\C}$.
There exist two non-isomorphic $\R$-models of $\SO_{3,\C}$\hs: the split model $\SO_{2,1}$
and the compact (anisotropic) model $\SO_{3,\R}$.

Consider the indefinite real quadratic form in three variables
\[F_{2,1}(x_1,x_2,x_3)=x_1^2+x_2^2-x_3^2,\quad x_i\in \R.\]
Set $G_0=\SO(F_{2,1})=\SO_{2,1}$.
We consider the affine quadric $Y^+_{2,1}\subset\A^3_\R$ given by the equation $F_{2,1}(x)=+1$,
and the affine quadric $Y^-_{2,1}\subset\A^3_\R$ given by the equation $F_{2,1}(x)=-1$.
The group $G_0=\SO_{2,1}$ naturally acts on  $Y^+_{2,1}$ and  $Y^-_{2,1}$, and
 the stabilizers in $G_{0,\C}$ of the points $(0,0,i)\in Y^+_{2,1}(\C)$ and $(0,0,1)\in Y^-_{2,1}(\C)$
are both equal to $\SO_{2,\C}$.
We see that  $Y^+_{2,1}$ and $Y^-_{2,1}$ are $\SO_{2,1}$-equivariant $\R$-models of $Y=G/H$.
It is well known that $Y^+_{2,1}(\R)$ is a hyperboloid of one sheet, hence it is connected,
while $Y^-_{2,1}(\R)$ is a hyperboloid of two sheets, hence it is not connected.
It follows that the $\R$-varieties $Y^+_{2,1}$ and $Y^-_{2,1}$
are two non-isomorphic $\SO_{2,1}$-equivariant $\R$-models of $Y=G/H$.

Now consider the positive definite real quadratic form in three variables
\[F_{3}(x_1,x_2,x_3)=x_1^2+x_2^2+x_3^2,\quad x_i\in \R.\]
Set $G_0=\SO(F_{3})=\SO_{3,\R}$.
We consider the affine quadric $Y^+_{3}\subset\A^3_\R$ given by the equation $F_{3}(x)=+1$,
and the affine quadric $Y^-_{3}\subset\A^3_\R$ given by the equation $F_{3}(x)=-1$.
The group $G_0=\SO_{3,\R}$ naturally acts on  $Y^+_3$ and  $Y^-_3$, and
the stabilizers in $G_{0,\C}$ of the points $(0,0,1)\in Y^+_3(\C)$ and $(0,0,i)\in Y^-_3(\C)$
are both equal to $\SO_{2,\C}$.
We see that $Y^+_{3}$ and $Y^-_{3}$ are $\SO_{3,\R}$-equivariant $\R$-models of $Y=G/H$.
Clearly, $Y^+_{3}(\R)$ is the unit sphere in $\R^3$, hence it is nonempty,
while $Y^-_{3}(\R)$ is empty.
It follows that the $\R$-varieties $Y^+_{3}$ and $Y^-_{3}$
are two non-isomorphic $\SO_{3,\R}$-equivariant $\R$-models of $Y=G/H$.
\end{example}

\begin{remark}
Let $k=\C$, $k_0=\R$, $\Gamma=\Gal(\C/\R)=\{1,\gamma\}$.
Let $Y=G/H$ be a spherical homogeneous space of a connected reductive group $G$ over $\C$,
and let $G_0$ be a real model of $G$ with the corresponding homomorphism $\sigma\colon \Gamma\to\SAut(G)$,
that is, with an anti-holomorphic involution $\sigma_\gamma\colon G\to G$.
Assume that there exists an anti-holomorphic  $\sigma_\gamma$-equivariant semi-automorphism
$\mu_\gamma\colon Y\to Y$  (such $\mu_\gamma$ exists if and only if
$\ve_\gamma$ corresponding to $\sigma_\gamma$ preserves the combinatorial invariants of $Y$).
Corollary 1 in Section 2.5 of Cupit-Foutou \cite{CF}
(which inspired this section of the present article)
claims that if, moreover, $G$ is semisimple and  $Y$ is spherically closed, then
(a) {\em any} such $\mu_\gamma$ is involutive, that is, $\mu_\gamma^2=1$, and
(b) such involutive $\mu_\gamma$ is {\em unique.}
Unfortunately, this corollary is erroneous.
Example \ref{ex:anti-CF} disproves (a) (see Remark \ref{r:anti-CF}),
and  Example \ref{ex:SO3} disproves (b) (see also Theorem \ref{t:Boris-Boris}).
However, it is true that there {\em exists} an involutive
anti-holomorphic  $\sigma_\gamma$-equivariant semi-automorphism $\mu_\gamma$\hs;
see Theorem \ref{t:CF-CF}.
\end{remark}

\section{Equivariant models of spherical embeddings\\
of automorphism-free spherical homogeneous spaces}
\label{s:embeddings}

In this section we assume that $\NGH=H$.

\begin{theorem}\label{t:Huruguen-bis}
Let $k,\ G,\ H,\ Y=G/H,\ k_0,\ \Gamma,\ G_0$ be as in Subsection \ref{ss:mt}, in particular, $\charr k=0$.
We assume that
\begin{enumerate}
\item[\rm(i)] $G_0$ is an inner form of a split group, and
\item[\rm(ii)] $\NGH=H$.
\end{enumerate}
Let $Y\into Y'$ be an {\emm arbitrary} spherical embedding of $Y$.
Then $Y'$ admits a  $G_0$-equivariant  $k_0$-model $Y'_0$.
This model is compatible with the $k_0$-model $Y_0$ of $Y$ from Theorem \ref{t:main}
and is unique up to a canonical isomorphism.
\end{theorem}

This theorem generalizes Theorem 1.2 of Akhiezer \cite{Akhiezer}, who considered the case $k_0=\R$.
Note that Akhiezer considered only the wonderful embedding of $Y$, while
we consider an arbitrary spherical embedding, so our result is new even in the case $k_0=\R$.

\begin{proof}
We show that a $k_0$-model  of $Y'$, if exists, is unique.
Indeed, let $Y'_0$ be such a $k_0$-model.
For $\gamma\in\Gamma:=\Gal(k/k_0)$, let $\mu'_\gamma\colon Y'\to Y'$
be the corresponding $\gamma$-semi-automorphism of $Y'$.
Since $Y$ is the only open $G$-orbit in $Y'$, it is stable under $\mu'_\gamma$ for all $\gamma\in\Gamma$.
Since $k_0$ is a field of characteristic 0, hence perfect, this defines a  $G_0$-equivariant $k_0$-model $Y_0$ of $Y$,
which is unique because $\NGH=H$ and hence, $\Aut^G(Y)=\{1\}$.
Since $Y$ is Zariski-dense in $Y'$, we conclude that the model $Y'_0$ of $Y'$ is unique.

We prove the existence.
By Theorem \ref{t:main}, $Y$ admits a unique  $G_0$-equivariant $k_0$-model $Y_0$.
The model $Y_0$ defines an action of $\Gamma$ on the finite set $\sD$;
see, for example, Huruguen \cite[2.2.5]{Huruguen}.
Namely, for every $\gamma\in\Gamma$ we have a $\sigma_\gamma$-equivariant $\gamma$-semi-automorphism $\mu_\gamma$,
which induces an automorphism $m_\gamma\colon\sD\to\sD$ covering $s_\gamma\colon \Omega \to\Omega$;
see \eqref{e:mu-m}.
We show that this action of $\Gamma$ on $\sD$ is trivial.
Indeed, since $\NGH=H$, by Corollary \ref{c:D-injective} the surjective map
\[\zeta\colon \sD\to \Omega\]
is bijective.
Since by assumption $G_0$ is an inner form, for all $\gamma\in\Gamma$ we have $\veg=1$, hence $s_\gamma=1$.
Thus $\Gamma$ acts trivially on  $\Omega$ and on $\sD$.

Let $\CF(Y')$ denote the colored fan of $Y'$  (see Knop \cite{Knop-LV} or Perrin \cite[Definition 3.1.9]{Perrin}\hs)
which is a set of colored cones $(\sC,\sF)\in \CF(Y')$, where $\sC\subset V$ and $\sF\subset\sD$.
We know that $\Gamma$ acts trivially on $V=\Hom_\Z(\sX,\Q)$ and on $\sD$.
It follows that for every $\gamma\in\Gamma$ and for every colored cone $(\sC,\sF)\in \CF(Y')$, we have
\begin{equation}\label{e:stable-CF}
\gams(\sC)=\sC,\quad \gams(\sF)=\sF.
\end{equation}
It follows that the colored fan $\CF(Y')$ is $\Gamma$-stable.
Moreover, it follows from \eqref{e:stable-CF} that the hypothesis
of Theorem 2.26 of Huruguen \cite{Huruguen} is satisfied, that is, $Y'$ has a covering
by $G$-stable and $\Gamma$-stable open quasi-projective subvarieties.
By that theorem $Y'$ admits a $G_0$-equivariant  $k_0$-model compatible with $Y_0$.
 \end{proof}

\begin{remark} Huruguen \cite{Huruguen} assumes that $Y_0$ has a $k_0$-point,
but he does not use that assumption.
\end{remark}

\begin{remark}
In Theorem \ref{t:Huruguen-bis} we do not assume that $Y'$ is quasi-projective.
\end{remark}

\appendix

\section{Algebraically closed descent\\ for spherical homogeneous spaces}
\label{s:app-AB}

The proofs in this appendix were communicated to the author by experts.

\begin{theorem}\label{t:AB}
Let $G_0$ be a  connected  reductive group defined over an
{\emm algebraically closed} field $k_0$ of characteristic 0.
Let $k\supset k_0$ be a larger algebraically closed field.
We set $G=G_0\times_{k_0} k$, the base change of $G_0$ from $k_0$  to $k$.
Let $H\subset G$ be a {\em spherical subgroup} of $G$ (defined over $k$).
Then $H$ is  conjugate to a (spherical) subgroup defined over $k_0$.
\end{theorem}

The proof is based on the following result of Alexeev and Brion \cite{AB}:

\begin{proposition} [\rm {\cite[Theorem 3.1]{AB}}]
\label{p:AB}
Let $G$ be a connected reductive group over an algebraically closed field $k$ of characteristic 0.
For any $G$-scheme $X$ of finite type, only finitely many
conjugacy classes of {\em spherical} subgroups of $G$ occur as isotropy groups of points
of $X$.
\end{proposition}

\begin{proof}[Proof of Theorem \ref{t:AB}]
The theorem will be proved in five steps.

1) Let $X_0$ be a variety equipped with an action of $G_0$. Then $X_0$
is the disjoint union of locally closed $G_0$-stable subvarieties $X_0^m$
consisting of all orbits of a fixed dimension $m$.
Note that the dimension $d(x)$ of the $G_0$-orbit of $x\in X_0$
is a lower semi-continuous function on $X_0$ (see, for  instance, Popov and Vinberg \cite[Section 1.4]{PV}\hs).
This means that for every number $\xi\in\R$, the subset
$\{x\in X_0\ |\ d(x)> \xi\}$
is open in $X_0$.
It follows that the union of the $G_0$-orbits in $X_0$ of maximal dimension is an open  subvariety.

2) Take for $X_0$ the variety of Lie subalgebras of $\Ggg_0 = \Lie\ G_0$ of a fixed
codimension, say $r$, and let $X$ be the $k$-variety obtained from $X_0$ by scalar extension.
Then $X$ is the variety of Lie subalgebras of codimension $r$ in $\Ggg = \Lie\ G$.
We write $x_\hh\in X$ for the point corresponding to a Lie subalgebra $\hh\subset \Ggg$,
and we write $\hh_x\subset\Ggg$ for the Lie subalgebra corresponding to a point $x\in X$.
The group $G$ acts on $X$ via  the adjoint representation in $\Ggg$, and
the stabilizer of a $k$-point $x_\hh$ in $X$ is the normalizer $\sN(\hh)$ of $\hh$  in $G$.
So the dimension of the orbit of $x_\hh $ is
\[\Ddim(G) - \Ddim \ \sN_G(\hh) = \Ddim(G) - \Ddim(\hh) - (\Ddim\ \sN_G(\hh) - \Ddim(\hh)\hs)
                 = r - \Ddim\, \nn_\Ggg(\hh)/\hh,\]
where $\nn_\Ggg(\hh)$ denotes the normalizer of $\hh$ in $\Ggg$.
Thus, if there exists $\hh$ such that  $\hh = \nn_\Ggg(\hh)$,
then the Lie subalgebras $\hh$ satisfying
this property correspond to the $k$-points of the open subset
consisting of the orbits of maximal dimension.
Note that $\nn_\Ggg(\hh)$ is the Lie algebra of $\sN_G(\hh)$.
So  if $\hh = \nn_\Ggg(\hh)$, then $\hh$ is an algebraic Lie algebra.

3) Let $H$ be a spherical subgroup of $G$ with Lie algebra $\hh$ such that $\nn_\Ggg(\hh) = \hh$.
We claim that the orbit $G\cdot x_\hh$ in $X$ is open.

First, note that the homogeneous $G$-variety $G\cdot x_\hh $ is spherical because the stabilizer of
$x_\hh $ in $G$ is the subgroup $\sN_G(\hh)$ with Lie algebra $\nn_\Ggg(\hh)=\hh=\Lie\, H$.
Hence for a suitable Borel subgroup $B \subset G$
we have $\Ddim (B\cdot x_\hh) = \Ddim (G\cdot x_\hh)$.
Consider the open subset
\begin{align*}
  \sU = \{x' \in X : \Ddim(B\cdot x') \ge \Ddim(B\cdot x_\hh) \} \subset X.
\end{align*}
Since by Step 2 the orbit  $G\cdot x_\hh$ has maximal dimension among the $G$-orbits in $X$,
for every $x'\in \sU$ we have
\begin{align*}
  \Ddim(B\cdot x') \ge \Ddim(B\cdot x_\hh) = \Ddim(G\cdot x_\hh) \ge \Ddim(G\cdot x'),
\end{align*}
hence
\[\Ddim (B\cdot x') = \Ddim (G\cdot x')=\Ddim(G\cdot x_\hh),\]
that is, if we write $\hh'=\hh_{x'}$\hs, then $\sN_G(\hh') \subset G$ is spherical
and $\hh' = \nn_\Ggg(\hh')$.
By Proposition \ref{p:AB} (due to Alexeev and Brion), the set of conjugacy
classes of spherical subgroups of the form $\sN_G(\hh_{x'})$ for $x'\in X$ is finite.
Hence, since $\hh_{x'}=\nn_\Ggg(\hh_{x'})$ is the Lie algebra of
$\sN_G(\hh_{x'})$ for every $x' \in \sU\subset X$, the set $G\cdot \sU$ contains
only finitely many $G$-orbits, which  are all of the same (maximal)
dimension. It follows that all these orbits are open; in particular,
the orbit $G\cdot x_\hh$ in $X$ is open.

4) By Step 3, the Lie algebras $\hh$ of spherical subgroups $H$ of $G$ such that
$\nn_\Ggg(\hh) = \hh$ form finitely many $G$-orbits, and the closures of these orbits
are irreducible components of the variety $X$, which is defined over $k_0$.
Since $k_0$ is algebraically closed,
it follows that every such orbit is defined over $k_0$ and, moreover,
 every such $G$-orbit has a $k_0$-point,
which proves the theorem for spherical subgroups such that $\nn_\Ggg(\hh) = \hh$.
Also
\[\sN_G(\hh) = \sN_G(H^0) = \sN_G(H),\]
where the latter equality follows from  Corollary \ref{c:Brion-Pauer} below.
Thus if $\sN_G(H)/H$ is finite, then $\sN_G(\hh)/H^0$ is finite,
and hence, $\nn_\Ggg(\hh) = \hh$.

5) To handle the case of an arbitrary spherical $k$-subgroup $H$ of $G$,
consider the spherical closure of $H$,
that is, the algebraic subgroup $\barH$ of $\NGH$ containing $H$ such that
\[\barH/H=\ker\left[\NGH/H\to\Aut\,\sD(G/H)\right].\]
By Corollary \ref{c:Luna} below, the spherical closure $\barH$ is spherically closed,
that is, $\sN_G(\barH)/\barH$ acts faithfully on the finite set of colors of $G/\barH$,
hence the group $\sN_G(\barH)/\barH$ is finite, and therefore, $\nn_\Ggg(\Lie\,\barH)=\Lie\,\barH$.
By Step 4 we may assume that $\barH$ is defined over $k_0$.
Now $H$ is an intersection of kernels of characters of $\barH$
(since the quotient $\barH/H$ is diagonalizable) and every such character is defined over $k_0$.
Thus $H$ is defined over $k_0$, as required.

An alternative proof, also based on Proposition \ref{p:AB} due to Alexeev and Brion,
is sketched in Knop's MathOverflow answer \cite{Knop-MO-AB}.
 \end{proof}

\begin{lemma}\label{l:Brion-Pauer}
Let $G$ be an abstract group and $H\subset G$ a subgroup.
Let $H_0\subset H$ be a characteristic subgroup of $H$
(this means that all automorphisms  of $H$ preserve $H_0$).
If the group $\sN_G(H_0)/H_0$ is abelian, then $\sN_G(H)=\sN_G(H_0)$.
\end{lemma}

\begin{proof}
Since $H_0$ is characteristic, we have $\sN_G(H)\subset \sN_G(H_0)$.
In particular, $H_0$ is normal in $H$, hence $H\subset \sN_G(H_0)$.
Consider the inclusions of groups
\[ H_0\subset H\subset \sN_G(H)\subset \sN_G(H_0)\]
and the inclusions of the corresponding quotient groups
\[H/H_0\subset \sN_G(H)/H_0 \subset \sN_G(H_0)/H_0\hs.\]
Since $\sN_G(H_0)/H_0$ is abelian, the subgroup $H/H_0$ is normal in $\sN_G(H_0)/H_0$\hs,
and hence, $H$ is normal in $\sN_G(H_0)$.
We see that $\sN_G(H_0)\subset \sN_G(H)$ and thus $\sN_G(H_0)=\sN_G(H)$.
 \end{proof}

\begin{corollary}[\rm {Brion and Pauer \cite[Corollary 5.2]{BP}}\hs]
\label{c:Brion-Pauer}
Let $G$ be a linear algebraic group over an algebraically closed field $k$ of characteristic 0.
Let $H\subset G$ be a spherical subgroup, and let $H^0$ denote the identity component of $H$.
Then $\sN_G(H)=\sN_G(H^0)$.
\end{corollary}

\begin{proof}
Clearly, $H^0$ is a characteristic subgroup of $H$.
Since $H$ is spherical, the subgroup $H^0$ is spherical as well, and therefore,
$\sN_G(H^0)/H^0$ is diagonalizable, hence abelian; see subsection \ref{ss:Losev-Aut} below.
Now the corollary follows from Lemma \ref{l:Brion-Pauer}.
 \end{proof}

From now on till the end of this appendix we follow  Avdeev \cite{Avdeev}.
Let  $G$ be a  connected  reductive group over an algebraically closed field $k$ of characteristic 0.
Fix a finite covering group $\Gtil\to G$ such that $\Gtil$ is a direct product of a torus with a simply
connected semisimple group. For every simple $\Gtil$-module $V$
(a finite dimensional irreducible representation $V$ of $\Gtil$),
 the corresponding projective
space $\ppP(V)$ has a natural structure of a $G$-variety.
Every $G$-variety arising in this way is said to be a simple projective $G$-space.

\begin{proposition}[\rm Bravi and Luna {\cite[Lemma 2.4.2]{BL}}]
{\rm (See also Avdeev {\cite[Cor\-ollary 3.24]{Avdeev}}.)}
\label{p:Avdeev}
For any spherical subgroup $H$ of a connected reductive group $G$
over an algebraically closed field $k$ of characteristic 0,
the spherical closure $\barH$ of $H$ is the common stabilizer in $G$ of all $H$-fixed points in all simple
projective $G$-spaces.
\end{proposition}

\begin{corollary}[\rm well-known]
\label{c:Luna}
Let $H$ be a spherical subgroup of a  connected  reductive group $G$
defined over an algebraically closed field $k$ of characteristic 0.
Let $H'$ denote the spherical closure of $H$,
and let $H''$ denote the spherical closure of $H'$.
Then $H''=H'$, that is, $H'$ is spherically closed.
\end{corollary}

This result was stated without proof in  Section 6.1 of Luna \cite{Luna-IHES}
(see also Avdeev \cite[Corollary 3.25]{Avdeev}).

\begin{proof}[Deduction of Corollary {\ref{c:Luna}} from Proposition {\ref{p:Avdeev}}]
Let $\ppP(V)$ be a simple projective $G$-space.
Let $\ppP(V)^H$ denote the set of fixed points of $H$ in $\ppP(V)$.
Since $H\subset H'$, we have $\ppP(V)^{H'}\subset \ppP(V)^H$.
By Proposition \ref{p:Avdeev} applied to $H$, we have $\ppP(V)^{H'}\supset \ppP(V)^H$.
Thus $\ppP(V)^{H'}= \ppP(V)^H$.

By  Proposition  \ref{p:Avdeev} applied to $H'$,
the group $H''(k)$ is the set of $g\in G(k)$ that fix $\ppP(V)^{H'}$
for all simple projective $G$-spaces $\ppP(V)$.
By  Proposition  \ref{p:Avdeev} applied to $H$,
the group $H'(k)$ is the set of $g\in G(k)$ that fix $\ppP(V)^H$
for all simple projective $G$-spaces $\ppP(V)$.
Since $\ppP(V)^{H'}=\ppP(V)^H$, we have $H''=H'$, as required.
 \end{proof}

\section{The action of the automorphism group on the colors\\ of a spherical homogeneous space}
\label{s:App}

\begin{center}
{\bf  By Giuliano Gagliardi}
\end{center}

In this appendix we prove Theorem \ref{t:Losev-3}, which we restate
below as Theorem~\ref{t:Losev-3-bis}. Our proof is based on Friedrich
Knop's MathOverflow answer \cite{Knop-MO} to Borovoi's question. Knop
writes that Theorem \ref{t:Losev-3-bis} was communicated to him by Ivan
Losev.

Let $G$ be a connected  reductive group over an algebraically closed
field $k$ of characteristic~$0$. Let $Y=G/H$ be a spherical
homogeneous space.

\begin{subsec}\label{ss:Losev-Aut}
  We use the notation of Section~\ref{s:invariants}. Let
  $\phi \in \Aut^G(Y)$ be a $G$-equivariant automorphism of $Y$. For
  every $\chi \in \sX$, the automorphism $\phi$ preserves the
  one-dimensional subspace $\smash{K(Y)^{(B)}_\chi}$ and thus acts on
  this space by multiplication by a scalar
  $a_{\phi,\chi}\in k^\times$. It is easy to see that we
  obtain a homomorphism
  \begin{align*}
    \kappa\colon \Aut^G(Y) &\to \Hom(\sX, k^\times)\text{,}\\
    \phi &\mapsto (\chi \mapsto a_{\phi,\chi})\text{.}
  \end{align*}
  The group $\Hom(\sX, k^\times)$ is naturally identified with the
  group of $k$-points of the $k$-torus with character group $\sX$.
  According to Knop \cite[Theorem~5.5]{Knop-Aut}, the homomorphism
  $\kappa$ is injective and its image is closed.
\end{subsec}

\begin{subsec} \label{r:NSR}
We present results of
Knop \cite{Knop-Aut} and
Losev \cite{Losev} describing $\Aut^G(Y)$.

 The uniquely determined set $\Sigma \subset \sX$ of
 linearly independent primitive elements $\gamma$
 of the lattice $\sX$  such that
  \begin{align*}
    \mathcal V = \bigcap_{\gamma \in \Sigma} \{v \in V : \langle v, \gamma\rangle \le 0\}\text{.}
  \end{align*}
  is called the set of \emph{spherical roots} of $Y$.
  Since the image  $\kappa(\Aut^G(Y))\subset \Hom(\sX, k^\times)$ is closed,
  this image corresponds to a sublattice $\Lambda \subset \sX$ such that
  \[\im\ \kappa=\{\phi \in \Hom(\sX, k^\times) \ |\ \phi(\chi)=1\text{ for all }\chi \in \Lambda\subset\sX\}.\]

  According to Losev \cite[Theorem~2]{Losev}, there exist integers
  $(c_\gamma)_{\gamma\in\Sigma}$ equal to $1$ or $2$ such that each
  $\gamma' := c_\gamma\cdot \gamma\in\sX$ is a primitive element of the lattice $\Lambda$.
  The set
  \[\SigmaN =\{c_\gamma\cdot\gamma\}_{\gamma\in\Sigma}\subset \Lambda\]
  generates the lattice $\Lambda$; see Knop \cite[Corollary~6.5]{Knop-Aut}.
  It follows that we have
  \[\Aut^G(Y) \cong \{\psi \in \Hom(\sX, k^\times) : \psi(\SigmaN) =
    \{1\}\}\text{.}\]
     Losev has shown how the coefficients $c_\gamma$ can be
  computed from the combinatorial invariants of~$Y$, but we shall only need the
  property recalled in Proposition~\ref{p:Losev-roots} below.
 For further  details, we refer to  Losev \cite{Losev}.
\end{subsec}

For $\alpha\in S$, let $\sD(\alpha)$ denote the set of colors $D\in \sD$
such that the parabolic subgroup $P_\alpha$ {\em moves} $D$,
that is, $\alpha\in \vs(D)$.
We need the following results of Luna \cite{Luna-GC}, \cite{Luna-IHES}:

\begin{proposition}  \label{p:facts-type-a}
  Let $\alpha \in S$.
  \setlength{\parskip}{0pt}
  \begin{enumerate}
  \item[\rm(1)] We have $|\sD(\alpha)| \le 2$.
   Moreover, $|\sD(\alpha)| = 2$ if and only if
  $\alpha \in \Sigma \cap S$.
 \item[\rm(2)] Assume $|\sD(\alpha)| = 2$ and write $\sD(\alpha) =
   \{D_{\alpha}^+, D_{\alpha}^-\}$. If $\rho(D_{\alpha}^+) =
   \rho(D_{\alpha}^-)$, then:
   \begin{enumerate}
  \item[\rm(i)] we have $\langle\rho(D_{\alpha}^+),\chi\rangle =
  \langle\rho(D_{\alpha}^-),\chi\rangle = \frac{1}{2}\langle\alpha^\vee,\chi\rangle$
  for all $\chi\in\sX$,
    where $\alpha^\vee\in\X_*(T)$ is the corresponding simple coroot;
  \item[\rm(ii)] we have $\varsigma(D_{\alpha}^+) = \varsigma(D_{\alpha}^-) = \{\alpha\}$.
  \end{enumerate}
  \end{enumerate}
\end{proposition}

\begin{proof}
  For (1), see Luna \cite[Sections 2.6 and 2.7]{Luna-GC} or Timashev
  \cite[Section 30.10]{Timashev}. For (2), we use that
  \cite[Theorem~2]{Luna-IHES} or \cite[Theorem~30.22]{Timashev}
  implies that the invariants of a spherical homogeneous space satisfy
  the axioms of a homogeneous spherical datum. These axioms are stated
  in \cite[Sections 2.1 and 2.2]{Luna-IHES} and
  \cite[Definition~30.21]{Timashev}. In particular, we have
  $\rho(D_{\alpha}^+) +\rho(D_{\alpha}^-) = \alpha^\vee|_{\sX}$
  and   for every $\beta \in \varsigma(D_{\alpha}^\pm)$
  we have $\beta \in \sX$ and $\langle \rho(D_{\alpha}^\pm), \beta \rangle = 1$.
   With the assumption $\rho(D_{\alpha}^+) = \rho(D_{\alpha}^-)$,
   we obtain (i) and then (ii).
 \end{proof}

We need the following result of Losev:

\begin{proposition} \label{p:Losev-roots}
 If $\alpha \in \Sigma \cap S$ and
  $\langle\rho(D_\alpha^+),\chi\rangle = \langle\rho(D_\alpha^-),\chi\rangle
    =\frac{1}{2}\langle\alpha^\vee,\chi\rangle$ for all $\chi\in\sX$,
    then $2\alpha \in \SigmaN$ (hence $\alpha \notin \SigmaN$).
\end{proposition}

\begin{proof} See Losev \cite[Theorem 2 and Definition 4.1.1(1)]{Losev}.  \end{proof}

The following theorem is the main result of this appendix:

\begin{theorem}[\rm Losev, unpublished]
\label{t:Losev-3-bis}
The homomorphism
\begin{equation*}
\Aut^G(Y)\to \Aut_\Omega(\sD)
\end{equation*}
 is surjective.
\end{theorem}

\begin{proof}
  Let $\sA$ denote the set of simple roots $\alpha \in S$ such that
  $|\sD(\alpha)| = 2$ and $\rho(D_{\alpha}^+) = \rho(D_{\alpha}^-)$.
  By Proposition~\ref{p:facts-type-a}, for every $\alpha\in\sA$ we
  have $\vs(D_{\alpha}^+) = \vs(D_{\alpha}^-) = \{\alpha\}$, hence the
  map $\alpha\mapsto \{D_\alpha^+, D_\alpha^-\}$ is a bijection
  between $\sA$ and the set of unordered pairs $\{D_\alpha^+,
  D_\alpha^-\}$ such that $(\rho \times \varsigma)(D_\alpha^+) = (\rho
  \times \varsigma)(D_\alpha^-)$.
  Note that there is a canonical bijection
  \[\sA\to \Omt,\quad \alpha\mapsto (\rho \times \varsigma)(D_\alpha^+).\]

  By Proposition \ref{p:facts-type-a}(1), for every $\alpha\in\sA$ we
  have $\alpha\in S\cap\Sigma\subset \sX$
  (because $\Sigma\subset  \sX$), and hence,
  there exists
  $f_\alpha \in \smash{K(Y)^{(B)}_\alpha}$ with $f_\alpha\neq 0$.
  Moreover, from Propositions \ref{p:facts-type-a} and
  \ref{p:Losev-roots} we obtain that $2\alpha \in \SigmaN$ (and
  $\alpha \notin \SigmaN$).

  We want to show that for every $\alpha \in \sA$ there exists
  $\phi_{\alpha}\in \Aut^G(Y)$ such that $\phi_{\alpha}$ swaps
  $D_{\alpha}^+$ and $D_{\alpha}^-$, but fixes all $D_{\beta}^+$
  and $D_{\beta}^-$ for $\beta\in \sA$,  $\beta\ne \alpha$.

  We have a natural homomorphism of algebraic $k$-tori
  \[ \Hom(\sX,k^\times)\to\Map(\Sigma, k^\times),\quad \psi\mapsto \psi|_\Sigma\hs, \]
  where $\Map(\Sigma, k^\times)$ denotes the group of maps $\Sigma\to k^\times$.
   Since the set $\Sigma \subset \sX$ is
  linearly independent, this homomorphism is surjective. It follows
  easily that any element of finite order in the group $\Map(\Sigma, k^\times)$
   can be lifted to an element of finite order in
  the group $\Hom(\sX,k^\times)$.

  Now let $\alpha \in \sA$.
  By the previous paragraph, there exists a homomorphism $\psi_{\alpha}\colon \sX \to k^\times$ with
  $\psi_{\alpha}(\alpha) = -1$ and $\psi_{\alpha}(\gamma) = 1$ for
  every $\gamma \in \Sigma \smallsetminus \{\alpha\}$, and such that
  $\psi_{\alpha}$ is of finite order in the group
  $\Hom(\sX,k^\times)$.
  Then we have
  $\psi_{\alpha}(\SigmaN) = \{1\}$. By
  \ref{r:NSR} there exists an automorphism of
  finite order $\phi_{\alpha} \in \Aut^G(Y)$ with
  \begin{align}\label{a:phi}
    \phi_{\alpha}(f_{\beta}) =
    \begin{cases}
      -f_{\beta} & \text {for $\beta = \alpha$,}\\
      f_{\beta} & \text{for $\beta \in \sA \smallsetminus \{\alpha\}$,}
    \end{cases}
  \end{align}
  where $f_\beta\in K(Y)^{(B)}_\beta$.
  Let $\stH \subset \NGH$ denote the subgroup containing $H$ such that
  \[
  \stH / H =\langle\phi_{\alpha}\rangle\subset\NGH/H=\Aut^G(Y),
  \]
  where $\langle\phi_{\alpha}\rangle$ denotes the finite subgroup
  generated by $\phi_{\alpha}$. We set $\stY = G/\stH$. We use the
  same notation for the combinatorial objects associated to the
  spherical homogeneous space $\stY$ as for $Y$, but with a tilde
  above the respective symbol. The morphism of $G$-varieties $Y\to
  \stY$ induces an embedding $K(\stY)\into K(Y)$, and $K(\stY)$ is
  the fixed subfield of $\phi_{\alpha}$. Since $K(\stY)$ is a
  $G$-invariant subfield of $K(Y)$, we have $\stsX \subset \sX$.

  By \eqref{a:phi} we have
  $\phi_{\alpha}(f_{\alpha})=-f_{\alpha}\neq f_{\alpha}$\hs.
  We see that $f_{\alpha} \notin K(\widetilde Y)$, hence $\alpha \in
  \sX \smallsetminus \stsX$; in particular $\alpha \notin \stSigma$.
  By Proposition~\ref{p:facts-type-a}(1) we have $|\stsD(\alpha)|
  \le 1$, hence the two colors in $\sD(\alpha)$ are mapped to one
  color by the map $Y \to \stY$, that is, $\phi_{\alpha}$
  swaps $D_{\alpha}^+$ and $D_{\alpha}^-$.

  On the other hand, for every $\beta \in \sA \smallsetminus \{\alpha\}$,
  by \eqref{a:phi} we have $\phi_{\alpha}(f_\beta)=f_\beta$,
  hence $f_\beta\in K(\stY)$ and $\beta\in\stsX$.
  Since $\beta$ is a primitive element of $\sX$, it is a primitive
  element of $\stsX\subset\sX$. The natural map $V \to
  \widetilde V$ induced by $Y \mapsto \stY$ is bijective and
  identifies $\sV$ and $\stsV$ (see Knop
  \cite[Section~4]{Knop-LV}). Since $\beta \in \Sigma$ is dual to a
  wall of $-\sV$, it is dual to a wall of $-\stsV=-\sV$.
   It follows that $\beta \in S \cap \stSigma$; hence $|\stsD(\beta)|
  = 2$, and the two colors in $\sD(\beta)$ are mapped to distinct
  colors under $Y \to \stY$, that is, $\phi_{\alpha}$ fixes
  $D_{\beta}^+$ and $D_{\beta}^-$.
 \end{proof}

\end{document}